\documentclass{article}

\usepackage[utf8]{inputenc}
\usepackage{verbatim} 
\usepackage{amsmath,amsthm,amsfonts,amssymb}
\usepackage{tikz}
\usepackage{color}
\usepackage{graphicx}
\usepackage{subcaption} 
\usepackage{enumerate}
\usepackage{tabularx}
\usepackage{float}
\usepackage{psvectorian}
\usepackage{lipsum}
\usepackage{fourier-orns}
\usepackage[T1]{fontenc}
\usepackage{currvita}

\usepackage{environ}

\DeclareMathOperator{\soc}{soc}
\DeclareMathOperator{\Aut}{Aut}

\DeclareMathOperator{\Cos}{Cos}

\DeclareMathOperator{\PSL}{PSL}

\DeclareMathOperator{\GL}{GL}
\DeclareMathOperator{\Sym}{Sym}
\DeclareMathOperator{\Supp}{Supp}

\DeclareMathOperator{\OG}{\mathcal{OG}}

\DeclareMathOperator{\B}{\mathcal{B}}

\theoremstyle{plain}

\newtheorem{Theorem}{Theorem}[section]
\newtheorem{theorem}{Theorem}[section]

\newtheorem{Proposition}[Theorem]{Proposition}
\newtheorem{Lemma}[Theorem]{Lemma}

\newtheorem{Problem}{Problem}
\newtheorem{Construction}{Construction}
\newtheorem{Definition}{Definition}

\theoremstyle{definition}

\newtheorem{Remark}[Theorem]{Remark}

\usepackage{blindtext}

\title{Basic Tetravalent Oriented Graphs with Cyclic Normal Quotients}
\author{Nemanja Poznanovi\'{c}\footnote{University of Melbourne, Australia.  email: {nempoznanovic@gmail.com}} \hspace{1cm} Cheryl E. Praeger\footnote{University of Western Australia, Australia. email:  {cheryl.praeger@uwa.edu.au}}}
\date{June 2022}

\begin{document}

\maketitle
\begin{abstract}
Let $\OG(4)$ denote the family of all graph-group pairs $(\Gamma, G)$ where $\Gamma$ is finite, 4-valent, connected, and $G$-oriented ($G$-half-arc-transitive). 
A subfamily of $\OG(4)$ has recently been identified as `basic' in the sense that all graphs in this family are normal covers of at least one basic member. 
In this paper we provide a description of such basic pairs which have at least one $G$-normal quotient which is  isomorphic to a cycle graph. In doing so, we produce many new infinite families of examples and solve several problems posed in the recent literature on this topic.
This result completes a research project aiming to provide a description of all basic pairs in $\OG(4)$.
\end{abstract}
\section{Introduction}\label{s:intro}

 A finite, simple, undirected graph $\Gamma$ is said to be \textit{$G$-oriented} (or $G$-half-arc-transitive) if some group $G \leq \Aut(\Gamma)$  is transitive on the vertices and edges of $\Gamma$, but is not transitive on the arcs. Equivalently, a graph $\Gamma$ is $G$-oriented if it is $G$-vertex- and $G$-edge-transitive and $G$ preserves some orientation $\Omega$ of the edge set of $\Gamma$.
Every $G$-oriented graph necessarily  has even valency (a result due to W.T. Tutte \cite{tutte1966connectivity}), and since the only valency two examples are disjoint unions of oriented cycles, the smallest `nontrivial' valency a $G$-oriented graph can have is four. 
These tetravalent graphs have been actively studied for several decades now, often in the hope that a good understanding of them may aid in the understanding of $G$-oriented graphs in general. The family $\OG(4)$ of graph-group pairs $(\Gamma,G)$ with $\Gamma$ a connected $G$-oriented graph of valency $4$ also has a special connection with the embedding of graphs into Riemann surfaces. Studies of such embeddings go back to work of Dyck~\cite{Dyck} in 1880 and Tutte \cite{Tutte49}  in 1949, with many recent papers inspired by the ground-breaking work of Jones and Singerman \cite{JS} (see  \cite[Section 2.2]{al2015finite} for a summary). 

A modern framework for studying the family $\OG(4)$ was proposed in \cite{al2015finite} and developed further in \cite{al2017cycle,al2017normal,poznanovic2021four}. This new approach aims to analyse $\OG(4)$ using a normal quotient reduction, a method which has been successfully used to study other families of graphs with prescribed symmetry conditions, see for instance \cite{morris2009strongly, praeger1993nan, praeger1999finite}. The
aim of this approach is to describe the family $\OG(4)$ in terms of graph quotients arising from normal subgroups of the groups contained in this family. 

Given a pair $(\Gamma, G)\in \OG(4)$, and a normal subgroup $N$ of $G$, the \textit{$G$-normal quotient} graph $\Gamma_N$ is defined as follows: the vertices of $\Gamma_N$ are the $N$-orbits on $V\Gamma$, and there is an edge between two vertices in $\Gamma_N$ if and only if there is an edge between vertices from the corresponding $N$-orbits in $\Gamma$ (see Section \ref{secOG4Normal} for details). The group $G$ also induces a group $G_N$ of automorphisms of $\Gamma_N$ so that we obtain another pair $(\Gamma_N, G_N)$. The  important result \cite[Theorem 1.1]{al2015finite} then tells us that  either the pair $(\Gamma_N, G_N) \in \OG(4)$ (and in this case $\Gamma$ is a $G$-normal cover of $\Gamma_N$), or  $\Gamma_N$ isomorphic to one of $K_1, K_2$ or $\mathbf{C}_r$ for some $r\geq 3$. In the latter case we say that the quotient is \textit{degenerate}.

Motivated by this result, we say that a pair $(\Gamma, G) \in \OG(4)$ is \textit{basic} if all of its $G$-normal quotients relative to non-trivial normal subgroups are degenerate. By \cite[Theorem 1.1]{al2015finite}, it follows that every member of $\OG(4)$ is a normal cover of (at least one) basic pair. Thus a good description of the basic pairs of $\OG(4)$ combined with a theory to describe the $G$-normal covers of these basic pairs should lead to a good description of the whole family $\OG(4)$. 

In \cite{al2015finite}, these basic pairs were further divided into three types, and a research program was proposed with the aim of obtaining a description of each of these three types of basic pairs. The three basic types are called quasiprimitive, biquasiprimitive and cycle type, and are specified according to the nature of their (degenerate) $G$-normal quotients, see Table~\ref{TabBasic}. The basic pairs of quasiprimitive and biquasiprimitive types have been successfully described in \cite{al2015finite, poznanovic2021four} using structure theorems for quasiprimitive and biquasiprimitive permutation groups in \cite{praeger1993nan, praeger2003finite}.  

A basic pair $(\Gamma, G) \in \OG(4)$ is of \textit{cycle type} if it has at least one $G$-normal quotient isomorphic to a cycle graph $\mathbf{C}_r$ for some $r\geq 3$.  Of the three basic types, the basic pairs of cycle type are the most difficult to analyse for a number of reasons. 
First, while there are structure theorems available for quasiprimitive and biquasiprimitive permutation groups, no such general theory is available here.
Also, unlike the other basic types, a basic pair $(\Gamma, G)\in\OG(4)$ of cycle type can have various non-isomorphic cyclic normal quotients. These quotients, despite all being cycles, may have different orders and may be $G$-oriented or $G$-unoriented depending on  the $G$-action on the quotient graph. It therefore becomes crucial to consider and to understand how the various cyclic normal quotients of a basic pair relate to each other.

%


In \cite{al2017cycle} the importance of considering so-called `independent' cyclic normal quotients was demonstrated. Two cyclic normal quotients of a pair ($\Gamma, G) \in \OG(4)$ are said to be independent if they are not extendable to a common cyclic normal quotient (see Section \ref{secCyclicNormal} for details). The existence of independent cyclic normal quotients of a pair $(\Gamma, G) \in \OG(4)$ proves to be restrictive enough to allow for a structural description of  such pairs. Each such pair $(\Gamma,G)$ was shown in \cite[Theorem 2]{al2017cycle} to have a particular normal quotient $(\overline{\Gamma}, \overline{G}) \in \OG(4)$ with $\overline{\Gamma}$ lying in one of six infinite families given in \cite[Table  1]{al2017cycle}. The underlying unoriented graphs appearing in these families are all either direct products of two cycles $\mathbf{C}_r \times \mathbf{C}_s$ for some $r,s\geq 3$, or are induced subgraphs or standard double covers of these direct product graphs. Moreover, if $(\Gamma, G)$ belongs to one of these  families then the group $G$ is metacyclic and the vertex stabiliser  $G_\alpha$ has order 2.
All  pairs of this type  have at least one cyclic normal quotient which is $G$-unoriented.  The basic pairs $(\Gamma, G)\in \OG(4)$ of cycle type, having independent cyclic quotients are determined in the upcoming article \cite{independent}.

The main result of this paper builds on these results by providing a description of the basic pairs $(\Gamma, G)\in \OG(4)$ of cycle type which \textit{do not} have independent cyclic normal quotients. Our main result is stated below.  Note that the graphs $\overline{\mathbf{C}}_r(2, t)$ appearing in the statement of Theorem  \ref{CycleMainTheorem} are defined in Definition \ref{defhighlyarc} and described in detail in \cite{praeger1989highly}.

\begin{theorem}\label{CycleMainTheorem}
	Suppose $(\Gamma, G) \in \OG(4)$ is basic of cycle type and does not have independent cyclic normal quotients. Then one of the following holds.
	\begin{enumerate}[1.]
       	\item All cyclic normal quotients of $(\Gamma, G)$ are $G$-oriented and either 
       		\begin{enumerate}[(a)]
        \item $\soc(G)$ is a $2$-group and $\Gamma$ is isomorphic to one of the graphs $\overline{\mathbf{C}}_r(2, t)$  for some $r\geq 3$ and $t \geq 1$; 
		\item or $G$ has a unique minimal normal subgroup $N \cong T^k$, where $k\geq 1$  and where $T$ is a simple group, $T\ncong \mathbb{Z}_2$, and  $\Gamma_N \cong \mathbf{C}_r$ for some $r \geq 3$. In this case
					\begin{enumerate}
						\item $k\leq r2^r$ if $N$ is  nonabelian; or
						\item $k \leq r$, if $N$ is abelian.
					\end{enumerate}
				\end{enumerate}
			\item All cyclic normal quotients of $(\Gamma, G)$ are $G$-unoriented and $G$ has a unique minimal normal subgroup $N \cong T^k$, where $k\geq 1$ and where $T$ is a simple group and  $\Gamma_N \cong \mathbf{C}_r$ for some $r \geq 3$. In this case
			\begin{enumerate}[i.]
			\item $k\leq 2r$ if $N$ is nonabelian; or
			\item $k \leq r-1$, if $N$ is abelian.
			\end{enumerate}
    \end{enumerate}
\end{theorem}
We note that Theorem \ref{CycleMainTheorem} is analogous to theorems \cite[Theorem 1.3]{al2015finite} and \cite[Theorem 1.1]{poznanovic2021four} which describe the basic pairs of quasiprimitive and biquasiprimitive types, respectfully. 
For basic pairs $(\Gamma, G)\in\OG(4)$ of these two types the group $G$  has a unique minimal normal subgroup $N \cong T^k$ where $T$ is a simple group and $k \leq 2$ if $G$ is quasiprimitive, and $k\leq 8$ if $G$ is biquasiprimitive. Notice that, for all basic pairs $(\Gamma, G) \in \OG(4)$ considered in Theorem \ref{CycleMainTheorem}, apart from the graphs  $\overline{\mathbf{C}}_r(2,t)$, $G$ also has a unique minimal normal subgroup isomorphic to $T^k$, with $T$ a simple group and $k\geq 1$. The value of $k$ here can be arbitrarily large, but since in the proofs we choose $r$ to be the length of the (unique) largest cyclic normal quotient,  it follows from Theorem~\ref{CycleMainTheorem} that $k$ is bounded above by a function of the length $r$ of the largest cyclic normal quotient of $\Gamma$. 

The proof of Theorem \ref{CycleMainTheorem} is given in Section \ref{secCycleProof} and relies on theory developed over the next several sections. In Section \ref{secCyclicNormal} we obtain some basic results concerning basic pairs of cycle type. We then separately analyse those basic pairs which have no independent cyclic normal quotients and whose only cyclic normal quotients are oriented (in Section \ref{secOri}) and unoriented (in Section \ref{secUnori}). 

In Section \ref{secConst} we will provide several interesting constructions for infinite families of basic pairs of cycle type. For each of the cases 1(b).i, 1(b).ii, 2.i, and 2.ii of Theorem \ref{CycleMainTheorem}, there is a corresponding construction in Section \ref{secConst} which shows that the  upper bound on the integer $k$ is tight. Notably, Construction \ref{NonabelianUnorientedConst2r} which is given in Section \ref{secConst} also provides an affirmative answer to a problem posed in \cite[Problem 1]{al2017cycle}, which asks if the number of unoriented cyclic normal quotients that a basic pair $(\Gamma,G)$ can have can be unboundedly large, see Theorem~\ref{thm:problemsolved}. We note also that Construction \ref{NonabelianOrientedConst24} for case  1(b).i of Theorem \ref{CycleMainTheorem} gives examples only for the case $r=3$, and we ask in Problem \ref{prob1}  whether or not the bound $k\leq r2^r$ in Theorem \ref{CycleMainTheorem}, case 1(b).i, is tight for $r\geq4$.

\section{Preliminaries}
All graphs in this paper are simple and finite. Given a graph $\Gamma$ we will always let $V\Gamma$ and $E\Gamma$ denote the sets of vertices and edges of $\Gamma$ respectively. Given a vertex $\alpha \in V\Gamma$ we will let $\Gamma(\alpha)$ denote the set of neighbours of $\alpha$. 
An \textit{arc} of a graph $\Gamma$ is an ordered pair of adjacent vertices. We will let $A\Gamma$ denote the set of arcs of $\Gamma$. Given an arc $(\alpha, \beta) \in A\Gamma$, we will call $(\beta,\alpha)$ its \textit{reverse arc}. In particular, each edge $\{\alpha,\beta\} \in E\Gamma$ has two arcs associated with it, namely $(\alpha,\beta)$ and $(\beta,\alpha)$. 

Given a group $G$ acting on a set $X$, we will always let $G^X$ denote the subgroup of $\Sym(X)$ induced by the group $G$. For an element $x \in X$, the \textit{point-stabiliser} $G_{x}$ in $G$ of $x$ is a subgroup of $G$ defined by
$G_{x}:=\{g \in G : x^g =x\}$. For a subset $\Delta \subseteq X$, the \textit{setwise stabiliser}  in $G$ of $\Delta$ is the subgroup $G_\Delta := \{g \in G : \Delta^g = \Delta\}$, where $\Delta^g = \{x^g : x \in \Delta\}$. 
Given an element $x \in X$, the set $x^G := \{x^g : g \in G\}$ is called a \textit{$G$-orbit}. The $G$-orbits form a partition of the set $X$.
A permutation group $G^X$ is said to be \textit{transitive} if all elements of $X$ lie in a single $G$-orbit, is said to be 
\textit{semiregular} if only the identity element of $G$ fixes a point in $X$, and is said to be \textit{regular} if it is semiregular and transitive. 

Given a group $G \leq \Aut(\Gamma)$, we say that $\Gamma$ is \textit{$G$-vertex-transitive}, \textit{$G$-edge-transitive}, or \textit{$G$-arc-transitive} if $G$ is transitive on $V\Gamma, E\Gamma$, or $A\Gamma$ respectively. Given a group $G$ acting on a graph $\Gamma$, and a vertex $\alpha \in V\Gamma$, the \textit{vertex stabiliser} $G_\alpha$ in $G$ of $\alpha$ is $G_\alpha := \{g\in G : \alpha^g = \alpha\}$, this is the point-stabiliser of $\alpha$ in $G^{V\Gamma}$.
We say that a  $G$-vertex-transitive graph $\Gamma$ is \textit{$G$-vertex-imprimitive} if $G^{V\Gamma}$ is imprimitive, and that $\Gamma$ is \textit{$G$-vertex-primitive} otherwise. 

For a detailed overview of these and other basic concepts of permutation group theory we refer the reader to \cite{praeger2018permutation}, and for other basic graph-theoretic concepts, please refer to \cite{godsil2013algebraic}.

\subsection{G-Oriented Graphs}\label{ssecBasic}

It is easy to see that a $G$-arc-transitive graph with no isolated vertices is necessarily both $G$-vertex- and $G$-edge-transitive. The converse however does not necessarily hold. As mentioned in the introduction, we will say that a graph $\Gamma$ is \textit{$G$-oriented}, with respect to a group $G\leq \Aut(\Gamma)$, if $G$ is transitive on $V\Gamma$ and $E\Gamma$ but is not transitive on $A\Gamma$.

In the literature, such graphs are also called \textit{$G$-half-arc-transitive}, or are said to admit \textit{a half-transitive $G$-action}. 
Suppose now that $\Gamma$ is a $G$-oriented graph and let $\Omega$ be a $G$-orbit on $A\Gamma$ (the arc set of $\Gamma$). Since $\Gamma$ is $G$-edge-transitive,  we know that one arc associated with each edge must lie in $\Omega$.   On the other hand, since $G$ is not transitive on $A\Gamma$, there must be another arc-orbit $\Omega'$ with the same property. Thus, $G$ will have exactly two (paired) orbits $\Omega$ and $\Omega'$ on $A\Gamma$, with exactly one arc associated with each edge contained in $\Omega$, and all reverse arcs contained in $\Omega'$.

It follows that if $\Gamma$ is a $G$-oriented graph, then no automorphism $g\in G$ may reverse an edge of $\Gamma$. Because of this, $\Gamma$ admits a natural $G$-invariant orientation of its edge set; simply take an arc-orbit, say $\Omega$, and orient an edge $\{x,y\}\in E\Gamma$ from $x$ to $y$ if and only if $(x,y) \in \Omega$. In particular, $\Omega$ can be seen as an orientation of $\Gamma$.

Technically speaking, a $G$-oriented graph $\Gamma$ as defined above is  `$G$-orientable' rather than $G$-oriented, as the orientation $\Omega$ has not been chosen yet. However this orientation is determined by the pair $(\Gamma, G)$ up to possibly reversing the orientation of all edges. Our notation suppresses the orientation $\Omega$, however when we need to refer to a fixed orientation of a given $G$-oriented graph $\Gamma$ we will specify this explicitly. The term $G$-oriented graph was suggested by B. D. McKay and has been used, for example in \cite{al2015finite,al2017cycle,al2017normal,poznanovic2021four,poznanovic2019biquasiprimitive}.

The discussion above suggests another (equivalent) approach to studying these graphs. Consider a transitive permutation group $G$ on a set $X$ and notice that the action of $G$ on $X$ induces a natural action on the set $X\times X$, namely $(x,y)^g =(x^g, y^g)$. The $G$-orbits in this induced action are called $G$-orbitals, and we say that an orbital $\Omega$ is \textit{self-paired} if whenever $(x,y) \in \Omega$, we also have $(y,x) \in \Omega$. If we take any $G$-orbital $\Omega$ which is \textit{not} self-paired, we may define a graph $\Gamma_\Omega$ with vertex set $X$ and with an edge $\{x,y\}$ if and only if $(x,y) \in \Omega$. If we orient all edges from $x$ to $y$ if  and only if $(x,y) \in \Omega$ then $\Gamma_\Omega$ will be a $G$-vertex- and $G$-edge-transitive graph and $G$ will preserve this edge orientation. All $G$-oriented graphs arise in this way. 

Finally note that we may also view $G$-oriented graphs as $G$-arc-transitive directed graphs, though it is preferable to view them as undirected graphs with an orientation induced by the group since a single graph $\Gamma$ can have many orientations induced by different groups. Still, we may always construct a $G$-oriented graph by taking the underlying graph of a $G$-arc-transitive directed graph. 

Since $G$-oriented graphs are vertex-transitive, they are necessarily regular. Moreover, all $G$-oriented graphs have even valency  \cite{tutte1966connectivity}. Hence for each even integer $m$ we define $\OG(m)$ to be the family of all graph-group pairs $(\Gamma, G)$ where $\Gamma$ is connected, $m$-valent and $G$-oriented. Note that we restrict our attention to connected graphs only since all connected components of a vertex-transitive graph are isomorphic. Further let $\mathcal{O G}:=\bigcup_{m \text { even }} \mathcal{O} \mathcal{G}(m)$ denote the family of all pairs $(\Gamma, G)$ where $\Gamma$ is  connected and $G$-oriented. Occasionally we will also need to discuss $G$-arc-transitive graphs. For each integer $k\geq 1$, let $\mathcal{A G}(k)$ denote the set of graph-group pairs $(\Sigma, H)$ where $\Sigma$ is a connected, $H$-arc-transitive graph of valency $k$. Also let $\mathcal{A G}:=\bigcup_{k \geqslant 1} \mathcal{A} \mathcal{G}(k)$.

Since the only finite, connected, 2-valent graphs are cycles, it is easy to see that these oriented cycles completely characterize the family $\OG(2)$. As mentioned earlier, the primary concern of this paper is the family $\OG(4)$.


\subsection{Vertex Stabilisers in $\OG(4)$}\label{secOG4stabilisers}

 Given a $G$-oriented graph $\Gamma$ with a fixed $G$-invariant edge-orientation, we will say that a vertex $\beta$ is an \textit{in-neighbour} of a vertex $\alpha$ if there is an edge $\{\alpha,\beta\}\in E\Gamma$ oriented from $\beta$ to $\alpha$, and we will say that $\beta$ is an \textit{out-neighbour} of $\alpha$ if there is an edge $\{\alpha,\beta\}\in E\Gamma$ oriented from $\alpha$ to $\beta$.
We will let $\Gamma_{in}(\alpha)$ and $\Gamma_{out}(\alpha)$ denote the sets of in- and out-neighbours of $\alpha$ respectively with respect to the fixed edge-orientation. In any $G$-invariant edge-orientation of a $G$-oriented graph, exactly half of the edges incident to any vertex $\alpha$  will be oriented from $\alpha$ to a neighbour, and the other half oriented in the opposite direction. Hence $\Gamma_{in}(\alpha)$ and $\Gamma_{out}(\alpha)$ partition $\Gamma(\alpha)$, and each contains half of the elements of $\Gamma(\alpha)$. Moreover, each of $\Gamma_{in}(\alpha)$ and $\Gamma_{out}(\alpha)$ is an orbit of the stabiliser $G_\alpha$ acting on $\Gamma(\alpha)$. 


An \textit{oriented $s$-arc} of a $G$-oriented graph is a sequence of vertices $(\alpha_0, \alpha_1,\dots,\alpha_s)$ of $\Gamma$, such that for each $i \in \{0,\dots,s-1\}$, $\alpha_i$ and $\alpha_{i+1}$ are adjacent, and each edge $\{\alpha_i, \alpha_{i+1}\}$ is oriented from $\alpha_i$ to $\alpha_{i+1}$.
By the above discussion, given a pair $(\Gamma, G) \in \OG(4)$, the stabiliser $G_\alpha$ of a vertex $\alpha \in V\Gamma$  will  have two orbits of length 2 on $\Gamma(\alpha)$. The following proposition gives us some other important information about such vertex stabilisers.
For a proof, see  for instance the first part of the proof of \cite[Lemma 6.2]{al2015finite}. 

    \begin{Proposition}\label{localprops}
	Let $(\Gamma, G) \in \OG(4)$ and let $\alpha  \in V\Gamma$. Let $s\geq 1$ be the largest integer such that $G$ acts transitively on the oriented $s$-arcs of $\Gamma$. Then the following hold:
	\begin{enumerate}[(i)]
	    \item   $G$ acts regularly on the set of oriented $s$-arcs of $\Gamma$, and
	    \item the stabiliser $G_\alpha $ is a group of order $2^s$ and is generated by $s$ involutions.
	\end{enumerate}
	\end{Proposition}

By Proposition \ref{localprops} we know that if $(\Gamma, G) \in \OG(4)$ then the  stabiliser $G_\alpha$ of a vertex $\alpha \in V\Gamma$ is a 2-group and is generated by involutions. In particular, these vertex stabilisers will always contain a subgroup isomorphic to $\mathbb{Z}_2$ and there are many interesting examples of pairs $(\Gamma, G) \in \OG(4)$ with $G_\alpha \cong \mathbb{Z}_2$. 
On the other hand, these vertex stabilisers may be arbitrarily large. In fact, the vertex stabilisers of pairs in $\OG(4)$ have been an active area of research for decades, see for example \cite{marusic2001point,potovcnik2010vertex,spiga2016order}. In particular, see \cite{spiga2019constructing,xia2019tetravalent} for some recent developments.

\subsection{Normal Quotients in $\OG(4)$}\label{secOG4Normal} 
Suppose that $\Gamma$ is a $G$-vertex-transitive graph and let $N$ be a nontrivial normal subgroup of $G$.
Let $\B_N$ be the partition of $V\Gamma$ into the $N$-orbits, that is, $\B_N = \{(\alpha^N)^g : g \in G\}$ where $\alpha$ is an arbitrary fixed vertex of $\Gamma$, and $\alpha^N$ denotes the $N$-orbit containing $\alpha$. In this case, $\B_N$ is a $G$-invariant partition of $V\Gamma$, so we may define the quotient graph $\Gamma_N := \Gamma_{\B_N}$ as described in Section~\ref{s:intro}. The graph $\Gamma_N$ defined in this way is  called the \textit{$G$-normal quotient} of $\Gamma$ with respect to $N$, and the group $G$ induces a group $G_N$ of automorphisms of $\Gamma_N$. It was shown in \cite{al2015finite}, that taking a normal quotient of a pair $(\Gamma, G) \in \OG(m)$ for $m\geq 4$, produces a connected $G$-vertex- and $G$-edge-transitive graph $\Gamma_N$ whose valency divides the valency of $\Gamma$.

Suppose now that $(\Gamma, G) \in \OG(4)$ and that $N$ is a normal subgroup of $G$. 
Specifically, $G_N = G/K$, where $K$ is the kernel of the $G$-action on $\Gamma_N$. By definition, $N\leq K$ and since $K$ fixes all $N$-orbits setwise, it follows that the $K$-orbits are the same as the $N$-orbits, so $\Gamma_K=\Gamma_N$. However, $K$ may be strictly larger than $N$.  

It was shown in \cite[Theorem 1.1]{al2015finite} 
that for any $(\Gamma, G) \in \mathcal{O} \mathcal{G}(4)$, and any nontrivial normal subgroup $N$ of $G$, either $\left(\Gamma_{N}, G_{N}\right)$ is also in $\mathcal{O} \mathcal{G}(4)$, or $\Gamma_{N}$ is isomorphic to $K_{1}$, $K_{2}$ or a cycle $\mathbf{C}_r$, for some $r \geq 3$. 
Note that if $(\Gamma_N, G_N)$ is itself a member of $\OG(4)$, that is, if $\Gamma_N$ is a 4-valent $G_N$-oriented graph, then $\Gamma$ is said to be a \textit{$G$-normal cover of $\Gamma_N$}.

Following the authors of \cite{al2015finite}, we  say that a pair $(\Gamma_N, G_N)$ is \textit{degenerate} if $\Gamma_N$ is isomorphic to one of $K_1$, $K_2$ or $\mathbf{C}_r$ for some $r \geq 3$. We also define a pair ($\Gamma, G) \in \OG(4)$ to be \textit{basic} if $(\Gamma_N, G_N)$ is degenerate relative to every nontrivial normal subgroup $N$ of $G$. 

As mentioned earlier, the basic pairs can be divided into types (quasiprimitive, biquasiprimitive and cycle type), depending on the possibilities for their normal quotient graphs.
Table \ref{TabBasic} gives a summary of these three types of basic pairs in $\OG(4)$. For more information and results on each of these basic types we refer the reader to the articles cited in the Reference column.

\begin{table}[H]
\centering
\resizebox{\linewidth}{!}{%
	\begin{tabular}{ l l l l }
		\hline
		Basic Type & Possible $\Gamma_N$ for $1\neq N \lhd G$ &  Conditions on $G$-action on $V\Gamma$ & Reference\\ 
		\hline
		Quasiprimitive & $K_1$ only & quasiprimitive & \cite{al2015finite} \\  
		Biquasiprimitive & $K_1$ and $K_2$ only & biquasiprimitive & \cite{poznanovic2019biquasiprimitive,poznanovic2021four} \\
		Cycle & At least one $\mathbf{C}_r $ $(r\geq 3)$ & at least one quotient action $D_{r}$ or $C_r$ & \cite{al2017cycle,al2017normal} \\
		\hline
	\end{tabular}}\caption{Types of Basic Pairs $(\Gamma, G) \in \OG(4)$.}\label{TabBasic}
\end{table}

\section{Cyclic Normal Quotients}\label{secCyclicNormal}
We will now begin working towards a proof of Theorem \ref{CycleMainTheorem} by investigating the basic pairs of cycle type. Throughout this paper, whenever we have a pair $(\Gamma, G) \in \OG(4)$ with $\Gamma_N$ a cyclic normal quotient, we will always let $\tilde{N}$ denote the normal subgroup of $G$ (containing $N$) which fixes the $N$-orbits setwise, that is, $\tilde{N}$ will always refer to the kernel of the $G$-action on the set of $N$-orbits on $V\Gamma$. Since $N \leq \tilde{N}$ and these groups have the same orbits on the vertices, it will always be the case that $\Gamma_N = \Gamma_{\tilde{N}}$.
 
  If $\Gamma_N$ is a cyclic normal quotient of $\Gamma$ then $\tilde{G} = G/\tilde{N}$ is isomorphic to either $C_r$ or $D_r$ (for $r\geq3$) \cite[Theorem 1.1]{al2015finite}, where $D_r$ denotes the dihedral group of order $2r$.
 If the normal quotient $(\Gamma_N, G/\tilde{N}$) of a pair $(\Gamma, G) \in \OG(4)$ is isomorphic to $(\mathbf{C}_r, C_r)\in \OG(2)$ for some $r\geq 3$, then we will say that the quotient is \textit{$G$-oriented} (or simply that it is oriented). If $(\Gamma_N, G/\tilde{N}$) is isomorphic to $(\mathbf{C}_r, D_r)\in \mathcal{AG}(2)$  then we will say that the quotient is \textit{$G$-unoriented} (or simply that it is unoriented).
 
 Of course, if $\Gamma_N$ is a $G$-oriented cyclic normal quotient, then the stabiliser $G_\alpha$ of any vertex will fix the $N$-orbit containing $\alpha$, and so $G_\alpha \leq \tilde{N}$. On the other hand, if $\Gamma_N$ is $G$-unoriented then any vertex $\alpha$ will have one out-neighbour (and one in-neighbour) in each of the two adjacent $N$-orbits, hence any non-identity automorphism which fixes $\alpha$ must swap at least two $N$-orbits, in particular $G_\alpha \cap \tilde{N} = 1$. 

 The following result on vertex stabilisers of pairs $(\Gamma, G) \in \OG(4)$ was first given in \textup{\cite[Lemma 2.1]{al2017cycle}}. We include its full statement here as it will be used heavily throughout this paper.
\begin{Lemma}\label{basics}
	Let $(\Gamma, G) \in \OG(4)$ have a cyclic normal quotient $\Gamma_N$, and let $\tilde{N}$ be the kernel of the $G$-action on the set of $N$-orbits on vertices. Let $\alpha \in V\Gamma$.
	\begin{enumerate}[(a)]
	\item If $\Gamma_N$ is $G$-unoriented, then $\tilde{N} = N$ is semiregular on $V\Gamma$, and $|G_\alpha| = 2$.
	\item If $\Gamma_N$ is $G$-oriented, then $\tilde{N}$ contains $G_\alpha$.
	\end{enumerate}
	Moreover, if $M$ is a normal subgroup of $G$ with at least three orbits such that $G_\alpha < M$,
	then the normal quotient $\Gamma_M$ is a $G$-oriented cycle.
\end{Lemma}

\begin{figure}[H]
	\begin{subfigure}{0.45\textwidth}
		\centering
		  \tikzstyle{every node}= [draw=none]
		\begin{tikzpicture}[scale =0.8]
		\node (max) at (0,2) {$G$};
		\node (b) at (0,1) {$\tilde{N}$};
		\node (d) at (-1,0) {$N$};
		\node (f) at (1,0) {$G_\alpha$};
		\node (min) at (0,-1) {$N_\alpha$};
		\node (minmin) at (0,-2) {$\{1\}$};
		\draw (minmin) -- (min) -- (d)  (max) -- (b) -- (f) (min) -- (f) (d) -- (b);
		\draw[preaction={draw=white, -,line width=6pt}] ;
		\end{tikzpicture}\subcaption{$\Gamma_N \cong \mathbf{C}_r$, $G$-oriented}
	\end{subfigure}
	\begin{subfigure}{0.45\textwidth}
		\centering
		\tikzstyle{every node}= [draw=none]
		\begin{tikzpicture}[scale =0.8]
		
		\node (b) at (0,2) {$G$};
		\node (d) at (-1,0) {$\tilde{N} = N$};
		\node (f) at (1,0) {$G_\alpha$};
		\node (min) at (0,-2) {$\{1\}$};
		
		\draw (min) -- (d)   (b) -- (f) (min) -- (f) (d) -- (b);
		\draw[preaction={draw=white, -,line width=6pt}] ;
		\end{tikzpicture}\subcaption{$\Gamma_N \cong \mathbf{C}_r$, $G$-unoriented}
	\end{subfigure}
	\caption{Subgroup lattice showing structure of $G$, depending on whether $\Gamma_N$ is $G$-oriented or $G$-unoriented.}
\end{figure}
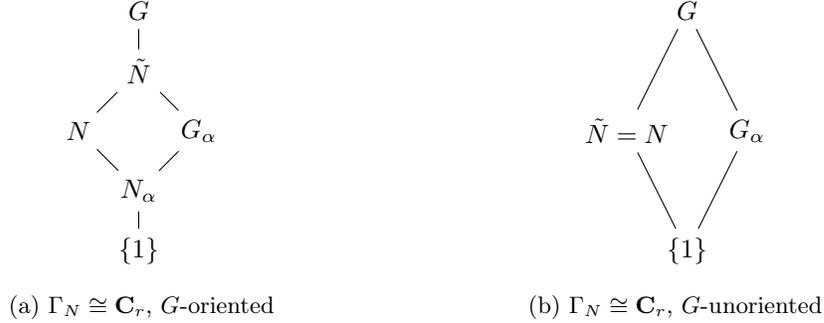

Since even a basic pair $(\Gamma, G)\in \OG(4)$ can have many different cyclic normal quotients, it is important to consider how these various quotients relate to each other. We will say that two cyclic   normal quotients $\Gamma_M$ and $\Gamma_N$ of $(\Gamma, G)\in \OG(4)$ are \textit{independent} if the normal quotient $\Gamma_K$, where $K= \tilde{M}\cap\tilde{N}$, is not a cycle. Note that if $(\Gamma, G)\in\OG(4)$ is basic of cycle type and $\Gamma_M$ and $\Gamma_N$ are independent cyclic normal quotients of $\Gamma$, then $K = \tilde{N}\cap\tilde{M} = 1$.

Hence there are two kinds of basic pairs of cycle type,  namely those with independent cyclic normal quotients and those without. A description of the pairs $(\Gamma, G)\in\OG(4)$ with independent cyclic quotients is given in \cite[Theorem 2]{al2017cycle}, and the basic pairs of independent-cycle type are determined in \cite{independent}. The objective of this paper is to describe the basic pairs which do not have independent cyclic normal quotients. 
\smallskip

Suppose now that $(\Gamma, G)\in \OG(4)$ is basic of cycle type and does not have independent cyclic normal quotients. Then for any two  cyclic normal quotients $\Gamma_M$ and $\Gamma_N$ of $\Gamma$, the quotient graph $\Gamma_K$, where $K= \tilde{M}\cap\tilde{N}$, must also be a cycle admitting an action of $G/K$.
Now $\tilde{N}/K$ is a normal subgroup of $G/K$, and moreover there is a $G$-invariant bijection between the $\tilde{N}$-orbits on the vertices of $\Gamma$, and the $(\tilde{N}/K)$-orbits on the vertices of $\Gamma_K$. Furthermore, this bijection induces a graph isomorphism, implying that $\Gamma_N =  \Gamma_{\tilde{N}} \cong (\Gamma_K)_{\tilde{N}/K}$ admitting a natural $G$-action, and the same is true of $\Gamma_M$. 
Thus since $\Gamma_K$ is also a cyclic normal quotient of $(\Gamma, G)$,  $\Gamma_M$ and $\Gamma_N$ are either both $G$-oriented or both $G$-unoriented cycles depending on whether $\Gamma_K$ is $G$-oriented or not. 

By the discussion above it follows that if $(\Gamma, G)$ is basic of cycle type  and does not have independent cyclic quotients then either all of its cyclic normal quotients are $G$-oriented, or all of its cyclic normal quotients are $G$-unoriented.
We may thus divide the basic pairs of cycle type into three distinct classes as follows. We will say that a basic pair $(\Gamma, G) \in \OG(4)$ with cyclic normal quotients is 
\begin{enumerate}
	\item basic of \textit{oriented-cycle type} if it does not have independent cyclic normal quotients \textbf{and} all of its cyclic normal quotients are $G$-oriented,
	\item basic of \textit{unoriented-cycle type} if it does not have independent cyclic normal quotients \textbf{and} all of its cyclic normal quotients are $G$-unoriented,
	\item basic of \textit{independent-cycle type} if it has independent cyclic normal quotients.	
\end{enumerate}

Since the basic pairs of independent-cycle type must necessarily have a $G$-unoriented cyclic normal quotient, see \cite[Theorem 2]{al2017cycle}, Lemma \ref{basics} ensures that if $(\Gamma, G)$ is  basic of unoriented-cycle type or of independent-cycle type then $|G_\alpha| =2$, and so the basic pairs of cycle type with larger vertex stabilisers must have only $G$-oriented quotients. The discussion above is summarized in Table \ref{cycletypestable}. This table also points to relevant sections where each of these classes is studied.

\begin{table}[H]
	\centering
	\resizebox{\linewidth}{!}{%
	\begin{tabular}{ l l l l }
		\hline
		Cycle Type & $G$-action on $\Gamma_N \cong \mathbf{C}_r$ &  $|G_\alpha|$ & Reference \\ 
		\hline
		Oriented-cycle type & $C_r$ on all cyclic quotients & unbounded & Section \ref{secOri}\\  
		Unoriented-cycle type & $D_r$ on all cyclic quotients & 2  & Section \ref{secUnori} \\
		Independent-cycle type & At least one cyclic quotient with $D_r$ action & 2 & \cite{al2017cycle} and \cite{independent}\\
		\hline
	\end{tabular}}\caption{Subdivision of basic pairs $(\Gamma, G) \in \OG(4)$ of cycle type.}\label{cycletypestable}
\end{table}
	
	\subsection{Describing the Socle}

	Our aim is to analyse the structure of $\soc(G)$ where $(\Gamma, G) \in \OG(4)$ is basic of cycle type. First note that if every minimal normal subgroup of $G$ has two or less orbits on $V\Gamma$, then every normal subgroup of $G$ will have at most two orbits. Thus every $G$-normal quotient of $\Gamma$ will be isomorphic to $K_1$ or $K_2$. Hence basic pairs of cycle type must have at least one minimal normal subgroup $N$ which has at least three orbits on $V\Gamma$. In fact, we will show that if $(\Gamma,G)$ is basic of cycle type then every minimal normal subgroup of $G$ has at least three orbits on $V\Gamma$. The next two lemmas establish this result.

\begin{Lemma}\label{notTrans}
	If $(\Gamma, G)\in \OG(4)$ is basic of cycle type then no minimal normal subgroup of $G$ is transitive on $V\Gamma$.
\end{Lemma}
\begin{proof}
	If $(\Gamma, G)\in \OG(4)$ is a basic pair of cycle type then $G$ has at least one minimal normal subgroup, say $N$, such that $\Gamma_N \cong \mathbf{C}_r$ with $r\geq 3$.
	Suppose now that $G$ also contains a transitive minimal normal subgroup $M$, and hence that $G= MG_\alpha$ for some $\alpha \in V\Gamma$. The existence of $N$ implies that $G$ is not simple and hence $M $ is a proper subgroup of $ G$. We now consider the normal subgroup $\tilde{N}$ containing $N$. Since $M$ is transitive and $\tilde{N}$ has $r$ orbits,  $M$ cannot be contained in $\tilde{N}$. Thus $\tilde{N}\cap M \neq M$ and so $\tilde{N}\cap M = 1$ by the minimality of $M$.
	
	If the graph $\Gamma_N$ is $G$-oriented, then $G_\alpha\leq \tilde{N}$.
	In this case, $M_\alpha = G_\alpha \cap M \leq \tilde{N}\cap M = 1$ so $M$ is semiregular.
	Therefore we have $G = M\rtimes G_\alpha$ and so $G_\alpha \leq \tilde{N} \cong \tilde{N}M/M \leq G/M\cong G_\alpha$.
	But this implies that $\tilde{N} = G_\alpha$, a contradiction.
	
	Hence the graph  $\Gamma_N$ is $G$-unoriented, and so  $|G_\alpha| = 2$ by Lemma \ref{basics}. In this case,  $M\cap G_\alpha = 1$, since  $G= MG_\alpha$ and $G \neq M$. In particular, $|G:M| =2$. But then $M\cap \tilde{N} = 1$ implies that $G = M\tilde{N} = M \times \tilde{N}$. Therefore $M \cong G/\tilde{N} \cong D_r$ for some $r\geq 3$, and since $D_r $ is not  characteristically simple  this contradicts the minimality of $M$.
	Hence $G$ contains no transitive minimal normal subgroups.
\end{proof}

\begin{Lemma}\label{3orbs}
	If $(\Gamma, G)\in \OG(4)$ is basic of cycle type then every minimal normal subgroup of $G$ has at least three orbits on $V\Gamma$.
\end{Lemma}
\begin{proof}
	
	By Lemma \ref{notTrans} it suffices to show that $G$ cannot have a minimal normal subgroup with two orbits on $V\Gamma$. Let $N$ be a minimal normal subgroup of $G$ such that $\Gamma_N \cong \mathbf{C}_r$ for some $r\geq 3$.
	
	Suppose now for the sake of a contradiction that $G$ has a minimal normal subgroup $M$ with two orbits on $V\Gamma$. Let  $\{\Sigma, \Sigma'\}$ denote the set of $M$-orbits and note that these two $M$-orbits form a $G$-invariant bipartition of $V\Gamma$ such that each part contains no edges, see for example \cite[Proposition 3.1]{al2015finite}. Notice in particular that $|\Sigma| >2$ since $|V\Gamma| = 2|\Sigma| >4$ as $\Gamma$ is 4-valent. As $M$ has two orbits on $V\Gamma$, it cannot be contained in $\tilde{N}$,  so $\tilde{N}\cap M \neq M$ and hence $\tilde{N}\cap M=1$ and $M\tilde{N} = M\times \tilde{N}$ by the minimality of $M$.
	
	This implies that  $M \cong M\tilde{N}/\tilde{N}$ which is normal in $G/\tilde{N}$. Furthermore, $M\tilde{N} / \tilde{N}$ must be minimal normal in $G/ \tilde{N}$ for if it were not then there would be a normal subgroup $X$ of $G$ such that $\tilde{N} < X < M\tilde{N}$.
    Each such normal subgroup $X$ has the form $X=Y\times \tilde{N}$ where $1<Y<M$ and $Y = X \cap M \unlhd G$, contradicting the minimality of $M$. 
	
	Hence $M\tilde{N} / \tilde{N} \cong M$ is a minimal normal subgroup of  $G/\tilde{N}$, and since $G/\tilde{N}$ is isomorphic to either $C_r$ or $D_r$ depending on whether $\Gamma_N$ is $G$-oriented or $G$-unoriented, we deduce that $M \cong \mathbb{Z}_p$ for some prime $p$ which divides $r$. Moreover, since $|\Sigma|$ divides $|M| = p$,  we have that $|M| = |\Sigma| = p$ and in particular $p$ is odd since $|\Sigma| >2$.
	Notice further that the number of $N$-orbits $r$ divides $|V\Gamma| = 2|\Sigma| = 2p$. So since $r \geq 3$ and each $N$-orbit contains at least two vertices, it follows that $r = p$ and hence every $N$-orbit contains exactly two vertices.
	
	Now since $|\Sigma|= p$ is odd and each $N$-orbit has length 2, we know that $\Sigma$ cannot be a union of $N$-orbits, and this implies that all $N$-orbits intersect each of $\Sigma$ and $\Sigma'$ in exactly one vertex. Let $\Delta_1 = \{\alpha, \alpha'\}$ and $\Delta_2 = \{\beta, \beta'\}$ be adjacent $N$-orbits in $\Gamma_N$, with $\alpha, \beta \in \Sigma$ and $\alpha'$, $\beta' \in \Sigma'$. Since each $N$-orbit is adjacent to exactly two other $N$-orbits in $\Gamma_N$ and $\Gamma$ is 4-valent, it follows that $\alpha$ is adjacent to both $\beta$ and $\beta'$. However the fact that  $\{\alpha,\beta\}\subset \Sigma$ and $\{\alpha,\beta\}$ is an edge of $\Gamma$ contradicts the fact that $\Sigma$ should contain no edges of $\Gamma$.  
	
	Hence no minimal normal subgroup of $G$ can have two vertex-orbits.
\end{proof}

We will now consider the structure of $\soc(G)$ in the two separate cases where $(\Gamma, G)$ is basic of oriented-cycle type (Section \ref{secOri}) and basic of unoriented-cycle type (Section \ref{secUnori}). The combined results obtained in these sections will provide a proof of Theorem \ref{CycleMainTheorem}.

\section{Oriented-Cycle Type}\label{secOri}
Recall that a basic pair $(\Gamma, G) \in \OG(4)$ with cyclic normal quotients is basic of oriented-cycle type if it has no independent cyclic normal quotients and all of its cyclic normal quotients are $G$-oriented.
In this section we will show that if $(\Gamma, G)$ is such a basic pair then this usually implies that $G$ has a unique minimal normal subgroup $N$, where $\Gamma_N$ is a $G$-oriented cycle, and all normal quotients of $\Gamma$ are simply quotients of the maximal cyclic quotient $\Gamma_N \cong \mathbf{C}_r$ for some $r \geq 3$. 

The possible exceptions to this rule are members of a well-understood family of graphs, which includes the lexicographic products $\mathbf{C}_r[2.K_1]$ for $r \geq 3$. These lexicographic product  graphs are described in detail in \cite[Construction 2.3]{al2017cycle} along with the maximal subgroup $G \leq \Aut(\mathbf{C}_r[2.K_1])$ for which $\Gamma = \mathbf{C}_r[2.K_1]$ is $G$-oriented. By \cite[Lemma 3.5]{al2015finite} and \cite[Lemma 2.4]{al2017cycle} such pairs $(\Gamma, G) \in \OG(4)$ are basic of oriented-cycle type. 

We will need to consider a larger family of  graphs discussed in   \cite{praeger1989highly}.  We briefly describe them here for completeness.

\begin{Definition}\label{defhighlyarc}
Let $v \geqslant 2, r \geqslant 3$ and $s \geqslant 1$ be integers. Define $\mathbf{C}_{r}(v, s)$ to be the digraph with vertex set $\mathbb{Z}_{r} \times \mathbb{Z}_{v}^{s}$ with $((i, \mathbf{x}),(j, \mathbf{y}))$ an edge, for $\mathbf{x}=\left(x_{1}, \ldots, x_{s}\right)$ and
	$\mathbf{y}=\left(y_{1}, \ldots, y_{s}\right) \in \mathbb{Z}_{v}^{s},$ if and only if $j=i+1$ and $\mathbf{y}=\left(y_{1}, x_{1}, \ldots, x_{s-1}\right)$.
\end{Definition}
\begin{Remark}
This definition comes from \cite[Definition 2.6]{praeger1989highly}. 	
 In what follows we will always  let $\overline{\mathbf{C}_{r}}(v, s)$ denote the underlying graph of  $\mathbf{C}_{r}(v, s)$.
 As mentioned in \cite[Remark 2.7]{praeger1989highly}, the graph $\overline{\mathbf{C}_{r}}(v, s)$ is isomorphic to the graph $\mathbf{C}(v, r, s)$ defined and studied in \cite{praegerxu1989}.
\end{Remark}

Note that by \cite[Theorem 2.8]{praeger1989highly}, the  digraphs $\mathbf{C}_r(v,s)$ are connected with valency $v$ and automorphism group $S_v \wr C_r$. When $v=2$ and a group $G$ acts transitively on the arcs of $\mathbf{C}_r(v,s)$, it follows that the underlying graph $\overline{\mathbf{C}}_r(v,s)$ is 4-valent and $G$-oriented. Hence we will be concerned with graphs of the form $\overline{\mathbf{C}}_r(2,s)$ for $r \geq 3$ and $s \geq 1$. Note the special case $\overline{\mathbf{C}}_r(2,1) \cong \mathbf{C}_r[2.K_1]$.

Our first  aim is to show that if  $(\Gamma, G)\in \OG(4)$ is basic of oriented-cycle type then $G$ has a unique minimal normal subgroup except possibly if $\Gamma  \cong \overline{\mathbf{C}}_r(2, s)$ for some $r\geq 3$ and $ s \geq 1$.  We begin with the following result concerning all basic pairs of this type. As usual, for $N$ a normal subgroup of a group $G$, we will let $\tilde{N}$ refer to the kernel of the $G$-action on the set of $N$-orbits on $V\Gamma$. We will use this notation consistently in this subsection.
\begin{Lemma}\label{OrSoc1}
	Suppose $(\Gamma, G)\in \OG(4)$ is basic of oriented-cycle type. Let $r\geq 3$ be the largest integer such that $\Gamma_N \cong \mathbf{C}_r$ for some normal subgroup $N$ of $G$.
	Then the following hold.
	\begin{enumerate}[(a)]
	    \item 	Every minimal normal subgroup of $G$ is contained in $\tilde{N}$.
	    \item If $L$ is a nontrivial normal subgroup of $G$ contained in $\tilde{N}$ then the $L$-orbits are equal to the $N$-orbits.
	\end{enumerate}
In particular, the orbits of every minimal normal subgroup of $G$ coincide with the $N$-orbits.
\end{Lemma}

\begin{proof}
			For part (a) take a minimal normal subgroup $M$ of $G$. By Lemma \ref{3orbs}, $\Gamma_M \cong \mathbf{C}_s$, with  $3\leq s\leq r$. Suppose that $M$ is not contained in $\tilde{N}$ so that $M \cap \tilde {N} = 1$. Since $\Gamma_N$ is $G$-oriented, it follows by Lemma \ref{basics} that $G_\alpha \leq \tilde{N}$ for $\alpha \in V\Gamma$. 
            Hence $M_\alpha=1$ for all $\alpha \in V\Gamma$. 
            Now consider $K = \tilde{M}\cap \tilde{N}$. Since $\Gamma_M$ is a $G$-oriented cycle, it follows again by Lemma \ref{basics} that $G_\alpha \leq \tilde{M}$. Hence $G_\alpha \leq \tilde{M}\cap \tilde{N} = K$, so $K \neq 1$ and therefore $\Gamma_K$ is a $G$-oriented cycle. Furthermore, we have $G_\alpha \leq K$, and since $M \cap \tilde {N} = 1$ we have $M\cap K = M \cap (\tilde{M}\cap \tilde{N})= 1$.
		
		Now  $MK \leq \tilde{M}$, so if we take $\alpha^M$ to be the $M$-orbit containing a vertex $\alpha$, and $\alpha^{K}$ to be the $K$-orbit containing $\alpha$, then using the facts that $M\cap K = 1$ and $M_\alpha$ = 1 we have that $|MK| = |M||K| =  |\alpha^M|\cdot|G_\alpha|\cdot|\alpha^{K}| \leq |\tilde{M}| = |\alpha^M|\cdot |G_\alpha|$. Hence $|\alpha^K| = 1$, a contradiction. Therefore $M \cap \tilde {N} \neq  1$ and $M \leq \tilde{N}$ as required. 
		
		For part (b) notice that if $L$ is a  normal subgroup of $G$ contained in $\tilde N$, it follows that the $L$-orbits are contained in the $\tilde N$-orbits. 	On the other hand, since $(\Gamma, G) $ is basic of oriented-cycle type, the maximality of $r$ implies that the $L$-orbits are equal to the $\tilde N $-orbits. 
\end{proof}

Thus every minimal normal subgroup of $G$ must be contained in a single normal subgroup $\tilde{N}$ where $\Gamma_N$ is the largest cyclic normal quotient of $\Gamma$. The next result shows that $G$ has a unique minimal normal subgroup unless possibly if $\tilde{N}$ is a 2-group.

\begin{Lemma}\label{OrientedMinimal}
	Suppose $(\Gamma, G)\in \OG(4)$ is basic of oriented-cycle type. Let $r\geq 3$ be the largest integer such that $\Gamma_N \cong \mathbf{C}_r$ for some normal subgroup $N$ of $G$. Then either $G$ has a unique minimal normal subgroup or $\tilde{N}$ is a $2$-group.
\end{Lemma}

\begin{proof}
	Consider a minimal normal subgroup $M$ of $G$. By Lemma \ref{OrSoc1}, $M$ is  contained in $\tilde{N}$ and the $M$-orbits coincide with the $\tilde{N}$-orbits.

	Suppose that $G$ has more than one minimal normal subgroup, and let $M_1$ and $M_2$ be any two distinct such subgroups. By Lemma \ref{OrSoc1} these subgroups are both contained in $\tilde{N}$ and the $M_1$-, $M_2$-,  and $\tilde{N}$-orbits are all equal.
	Let $\Delta$ be an $\tilde{N}$-orbit on $V\Gamma$. If we consider a vertex $\alpha \in \Delta$, then since by Lemma \ref{basics} we have $G_\alpha \leq \tilde{N}$, we know that $|\tilde{N}| = |\Delta||G_\alpha|$. 
	Moreover $|M_1|= |\Delta||(M_1)_{\alpha}|$ and $|M_2|= |\Delta||(M_2)_{\alpha}|$, and since $G_\alpha$ is a 2-group, both $(M_1)_{\alpha}$ and  $(M_2)_{\alpha}$ are 2-groups.
	
	Now $M_1M_2 \leq \tilde{N}$, and since $M_1\cap M_2 = 1$ we have $|M_1M_2| = |M_1||M_2| = |\Delta|^2|(M_1)_{\alpha}|\cdot|(M_2)_{\alpha}|$ divides $|\tilde{N}| = |\Delta||G_\alpha|$. Hence  $|\Delta||(M_1)_{\alpha}||(M_2)_{\alpha}|$ divides $|G_\alpha|$ and since $|\Delta|>1$, we get that $|\Delta|$ is a power of 2.
	Then since $|\tilde{N}| = |\Delta||G_\alpha|$, it follows that $\tilde{N}$ is a 2-group.
\end{proof}

We will now show that if $\tilde{N}$ (as defined Lemma \ref{OrientedMinimal}) is a 2-group, then $\Gamma$ is isomorphic to the underlying graph of one of the digraphs in Definition \ref{defhighlyarc}.

\begin{Lemma}\label{Oriented2group}
	Suppose $(\Gamma, G)\in \OG(4)$ is basic of oriented-cycle type with $|G_\alpha| = 2^s$ for $s \geq 1$. Let $r\geq 3$ be the largest integer such that $\Gamma_N \cong \mathbf{C}_r$ for some normal subgroup $N$ of $G$. If $\tilde{N}  $ is a $2$-group then $\Gamma \cong  \overline{\mathbf{C}}_r(2, t)$ where $1 \leq t \leq r-s$ . \end{Lemma}
	\begin{proof}
	Under these assumptions  $G = \tilde{N} \langle \sigma \rangle$ for some $\sigma \in G\backslash \tilde{N}$ where $\sigma^r \in \tilde{N}$. By Lemma \ref{OrSoc1} we know that all minimal normal subgroups of $G$ are contained in $\tilde{N}$. In particular, $\soc(G)$ is contained in $\tilde{N}$ and hence $\soc(G) \cong \mathbb{Z}_2^k$ for some $k\geq 1$ since $\tilde{N}$ is a 2-group.
	
	Let $Z(\tilde N)$ denote the centre of $\tilde N$. Since this is a nontrivial characteristic subgroup of $\tilde N$, it follows that $Z(\tilde N)$ is  a nontrivial normal subgroup of $G$. 
	It follows that $L := \soc(G) \cap Z(\tilde N)$ is a nontrivial normal subgroup of $G$. Now let $H$ be any minimal normal subgroup of $G$ contained in $L$. In particular, $H \cong \mathbb{Z}_2^\ell$ for some $\ell \geq 1$.  Since $H$ is a  normal subgroup of $G$ contained in $\tilde N$ it follows that the $H$-orbits are equal to the $\tilde N$-orbits by Lemma \ref{OrSoc1}.
	
	Now take a vertex $\alpha \in V\Gamma$ and let $\Delta$ denote the $\tilde{N}$-orbit containing $\alpha$.  Since $\Gamma_N$ is an oriented cycle,   $G_\alpha \leq \tilde{N}$ by Lemma	\ref{basics}. Now take $\beta \in \Delta$ and $h \in H$ such that $\alpha^h = \beta$. 
	Notice that since $H \leq Z(\tilde N)$ it follows that $[H, G_\alpha] = 1$. In particular, $G_\alpha = G_\alpha^h = G_{\alpha^h} = G_\beta$, and since $\Delta$ is an $H$-orbit, it follows that $G_\alpha$ fixes $\Delta$ pointwise. It follows that $G_\alpha$ is equal to the kernel of the $\tilde{N}$-action on $\Delta$ and so $G_\alpha $ is normal in $\tilde N$. 
	
Hence we may write $\tilde{N} = HG_\alpha$. By the previous paragraph we know that $\tilde{N}^\Delta$ is regular and in fact $\tilde{N}^\Delta \cong H^\Delta$ is elementary abelian. Now let $\{\Delta_1,
	\dots, \Delta_r\}$ denote the set of $N$-orbits on $V\Gamma$. By \cite[Theorem 5.5]{praeger2018permutation} we may view $G = \tilde{N}\langle \sigma \rangle$ as a subgroup of $\tilde{N}^\Delta \wr C_r$ with $\tilde{N}$ isomorphic to a subgroup of $\tilde{N}^{\Delta_1}\times \cdots \times \tilde{N}^{\Delta_r} \cong H^{\Delta_1}\times \cdots \times H^{\Delta_r}$. In particular, $\tilde{N}$ is elementary abelian and so $G_\alpha \cong \mathbb{Z}_2^s$ for some $s \geq 1$. Thus $\tilde{N}$ is an elementary abelian 2-group which does not act semiregularly on $V\Gamma$.

	
	Now consider one of the two paired $G$-arc-orbits on $\Gamma$ and take the orbital digraph $\Gamma'$ corresponding to this orbit so that  $\Gamma$ is the underlying graph of $\Gamma'$. Since $|G_\alpha| = 2^s$, we know that $G$ is transitive on the oriented $s$-arcs of $\Gamma$ by Proposition \ref{localprops} and hence $\Gamma'$ is $(G,s)$-arc-transitive with $\tilde{N}$ an abelian normal subgroup of $G$ which does not act semiregularly on $V\Gamma$. 
	We may thus apply \cite[Theorem 2.9]{praeger1989highly} directly to get that $\Gamma \cong \overline{\mathbf{C}}_{r'}(2, t)$ where $r' = |G: G^*|$ for $G^* = \langle G_\alpha^g : g \in G\rangle$ and $1\leq t \leq r' -s$. 
	
	 Left to show is that $r = r'$. We do this by showing that $G^* = \tilde{N}$. First note that since $\tilde{N}$ normalises $G_\alpha$ (as we showed above) we find that $G^* = \langle G_\alpha, G_\alpha^{\sigma}, \dots, {G_\alpha}^{\sigma^{r-1}}\rangle$. In particular, since $G_\alpha^{\sigma^i} \leq \tilde{N}^{\sigma^{i}} = \tilde{N}$ for each $i$, it follows that $G^* \leq \tilde{N}$.	On the other hand, since $G^*$ is a normal subgroup of $G$ contained in $\tilde{N}$, it follows by the maximality of $r$ that the $G^*$-orbits are equal to the $\tilde{N}$-orbits. Hence $|G^*| = |\Delta||G_\alpha| = |\tilde{N}|$. Thus $G^* = \tilde {N}$ and this completes the proof. 
	\end{proof}

  In summary, Lemmas \ref{OrSoc1} - \ref{Oriented2group} give the following result concerning basic pairs of oriented-cycle type.
	\begin{Proposition}\label{OrientedTheorem}
		Suppose that $(\Gamma, G) \in \OG(4)$ is basic of oriented-cycle type. Then either
		\begin{enumerate}[(a)]
			\item $\Gamma \cong \overline{\mathbf{C}}_r(2, t)$ as in Definition $\ref{defhighlyarc}$, for some $r\geq 3$ and $t \geq 1$, or
			\item $G$ has a unique minimal normal subgroup $N \cong T^k$ for some simple group $T \ncong \mathbb{Z}_2$, and $k\geq 1$. 
		\end{enumerate}
	\end{Proposition}
\begin{proof} 
	Lemmas \ref{OrientedMinimal} and \ref{Oriented2group}  together imply that either $G$ has a unique minimal normal subgroup or that $\Gamma \cong \overline{\mathbf{C}}_r(2, t)$, for some $r\geq 3$ and $t \geq 1$. If $G$ has a unique minimal normal subgroup $N \cong T^k$ with $T\cong \mathbb{Z}_2$ and $k\geq 1$, then any $N$-orbit $\Delta$ has length a 2-power. Hence we get that $|\tilde{N}| = |\Delta||G_\alpha|$ is a 2-power so $\tilde{N}$ is a 2-group.  Lemma \ref{Oriented2group} then implies that $\Gamma \cong \overline{\mathbf{C}}_r(2, t)$, for some $r\geq 3$ and $t \geq 1$. Thus we may assume that  $T\ncong \mathbb{Z}_2$ in case (b).
\end{proof}
	Notice that Cases (a) and (b) of Proposition \ref{OrientedTheorem} correspond directly to the Cases (a) and (b) of  Theorem \ref{CycleMainTheorem} Case 1.
	For the remainder of this section we will analyse those pairs $(\Gamma, G) \in \OG(4)$  which are basic of oriented-cycle type as described in  case (b) of Proposition \ref{OrientedTheorem}.  We will separately consider the cases where $N = \soc(G)$ is nonabelian or abelian over the next two subsections. In both cases we are able to bound the integer $k$ by a function of $|\Gamma_N| = r \geq 3. $

    Next we give an important lemma concerning pairs described by Case (b) of Proposition \ref{OrientedTheorem}.
    \begin{Lemma}\label{OrientedFaithful}
Suppose that  $(\Gamma, G)$, $N$ and $k$ are as in Proposition $\ref{OrientedTheorem}(b)$. Then $\tilde{N} = NG_\alpha$ and $\tilde{N}$ is faithful on $\Delta$. 
\end{Lemma}
\begin{proof}
Under these assumptions  $G = \tilde{N} \langle \sigma \rangle$ for some $\sigma \in G\backslash \tilde{N}$ where $\sigma^r \in \tilde{N}$. Consider an $N$-orbit $\Delta$ containing a vertex $\alpha$. Since  $\tilde{N}$ contains $G_\alpha$ by Lemma \ref{basics}, and since both $\tilde{N}$ and $N$ are transitive on $\Delta$  and $N \leq \tilde{N}$ it follows that $\tilde{N} = NG_\alpha$. 

We will now show that $\tilde{N}$ is faithful on $\Delta$. 
	Let $K$ be the kernel of the action of $\tilde{N}$ on $\Delta$. Then $K$ is contained in $G_\alpha$ implying that $K$ is a 2-group. Now since $N \cap K$ is also a 2-group which is normal in $N = T^k$ with $|T|\neq 2$, it follows that $N\cap K = 1$. 
	
If $K \neq 1$ then let $L:= $ O$_2(\tilde{N})$ denote the intersection of all Sylow 2-subgroups of $\tilde{N}$, and notice that $L$ contains $K \neq 1$. Since $L$ is characteristic in $\tilde{N}$, it follows that $L$ is normal in $G$. 
But then $G$ contains a nontrivial normal 2-subgroup and hence a minimal normal subgroup which is a 2-group. This contradicts Proposition \ref{OrientedTheorem}(b), so $K = 1$ and  $\tilde{N}$ is faithful on $\Delta$.
\end{proof}

We conclude this section by stating and proving an important fact regarding vertex stabilisers in pairs $(\Gamma, G)\in\OG(4)$  of oriented-cycle type. This fact follows from a much more general result given for instance in \cite[Lemma 6]{potovcnik2017smallest}. The original proof of this result relies on a more general normal quotient machinery than the one which we have been using. For completeness, we therefore add a proof which only involves normal quotients as defined in this paper.

\begin{Lemma}\label{Abelianstabiliser}
    Suppose that  $(\Gamma, G)$, $N$ and $k$ are as in Proposition $\ref{OrientedTheorem}(b)$ with $\Gamma_N \cong \mathbf{C}_r$ and $r\geq 3$. If $N$ is semiregular on $V\Gamma$ then $G_\alpha \cong \mathbb{Z}_2^s$ where $s\leq r$. 
    \end{Lemma}
\begin{proof}
	Fix a $G$-invariant orientation of the edges of $\Gamma$. Let $\{\Delta_0,\dots, \Delta_{r-1}\}$ denote the $N$-orbits on $V\Gamma$ and assume that all edges of $\Gamma$ are oriented from  $\Delta_i$ to $\Delta_{i+1}$ for $i \in \mathbb{Z}_r$ and $r\geq 3$. Take a vertex $\alpha \in \Delta_0$. 
	
	We will now consider the set of $N$-orbits on the arcs in this orientation of  $\Gamma$. Call this set $X$. Our first claim is that $|X| = 2r$. 
	Let $\beta$ and $\beta'$ be the two out-neighbours of $\alpha$ and note that both of these vertices are  contained in the adjacent $N$-orbit $\Delta_1$.  Since $N$ acts semiregularly on $V\Gamma$ it follows that the arcs $(\alpha, \beta)$ and $(\alpha , \beta')$ lie in different $N$-arc-orbits say $\Omega_1$ and $\Omega_2$. Furthermore since $\Delta_0$ is an $N$-orbit, it is clear that for any $\alpha' \in \Delta_0$,  one of the two arcs emanating from $\alpha'$ will lie in $\Omega_1$ and the other will lie in $\Omega_2$.  Thus all arcs from $\Delta_0$ to $\Delta_1$ will lie in one of the two $N$-arc-orbits $\Omega_1$ or $\Omega_2$. 
	The same reasoning tells us that if we take any two adjacent $N$-orbits $\Delta_j$ and $\Delta_{j+1}$ for $j \in \mathbb{Z}_r$, then all arcs from $\Delta_j$ to $\Delta_{j+1}$ will lie in one of two $N$-arc-orbits, say $\Omega_{2j+1}$ and $\Omega_{2j+2}$.  We can therefore see that $N$ has 2 arc orbits for each edge of $\Gamma_N = \mathbf{C}_r$. Hence $N$ has $2r$ orbits on the arcs of $\Gamma$ and $|X| =2r$. 
	
	Now $N$ is a normal subgroup of $\tilde{N}$, so we may consider the action of $\tilde{N}$ on $X$ where $\tilde{N} = NG_\alpha$ by Lemma \ref{OrientedFaithful}. Since $N$ is semireglar, we have $N\cap G_\alpha=1$ and $G_\alpha\cong \tilde{N}/N$.  First notice that each $\tilde{N}$-orbit on $X$ has size 2. To see this, notice that if we take the two adjacent $N$-orbits  $\Delta_0$ and $\Delta_1$, then all arcs from $\Delta_0$ to $\Delta_1$ lie in a single $\tilde{N}$-orbit.
	This is because any element $g \in G_\alpha \leq \tilde{N}$ which fixes $\alpha$ and swaps $\beta$ and $\beta'$, also swaps the  $N$-arc-orbits $\Omega_1$ and $\Omega_2$. The same is true for the arcs between any two adjacent $N$-orbits $\Delta_i$ and $\Delta_{i+1}$.  In other words, $\tilde{N}$ has exactly $r$ orbits on the arcs of $\Gamma$ and each one is a union of two $N$-arc-orbits on $\Gamma$. Thus each $\tilde{N}$-orbit on $X$ has size 2. 
	
	Let $K$ be the kernel of the action of $\tilde{N}$ on $X$. We claim that $K = N$. Of course we know that $N \leq K$ since each element of $X$ is an $N$-orbit. Hence $K = NK_\alpha$. Now consider an element  $g \in K_\alpha$. Since $g$ fixes $\alpha$ and also fixes (setwise) both of the arc-orbits $\Omega_1$ and $\Omega_2$, it follows that $g$ fixes both of the out-neighbours $\beta$ and $\beta'$ of $\alpha$. The same reasoning then implies that $g$ fixes pointwise the out-neighbours of $\beta$ and $\beta'$ and by induction and using the fact that $\Gamma$ is connected it follows that $g$ fixes all vertices of $\Gamma$, so $g=1$. Thus $K=N$. 
	
 Now we have that $\tilde{N} / N \cong G_\alpha$  acts faithfully on $X$. Moreover since each $\tilde{N} / N$-orbit on $X$ has size 2. It follows that $G_\alpha \cong  \tilde{N} / N$ is a subgroup of $S_2^r \cong \mathbb{Z}_2^r$.  
\end{proof}

\subsection{Oriented-Cycle Type: nonabelian socle}

We begin by analysing the case where Proposition \ref{OrientedTheorem}(b) holds and $N \cong T^k$ for some nonabelian simple group $T$ and $k \geq 1$. Our aim is to consider the possible values for the integer $k$, the number of simple direct factors of $N$. 
Our first result in this direction distinguishes two important cases concerning the structure of the unique nonabelian minimal normal subgroup  $N$ of $G$. 

\begin{Lemma}\label{OrientedNonabelianSoc1}
	Suppose that  $(\Gamma, G)$, $N, T$ and $k$ are as in Proposition $\ref{OrientedTheorem}(b)$ with $\Gamma_N \cong \mathbf{C}_r$, $r \geq 3$, and $N$ nonabelian. Then either
	\begin{enumerate}[(a)]
		\item $N$ is the unique minimal normal subgroup of $\tilde{N}$, or
		\item $\tilde{N}$ has $d$ minimal normal subgroups for some $d$ dividing  $r$ with $d>1$. Each of these subgroups is isomorphic to $T^\ell$ for some $\ell \geq 1$, and $N$ is the direct product of these subgroups. In particular,  $k = d\ell$ where $\ell \geq 1$ and $d$ divides $r$.
	\end{enumerate}
In either case all minimal normal subgroups of $\tilde{N}$ lie in $N$.
\end{Lemma}

\begin{proof}
	The group $N$ in Proposition \ref{OrientedTheorem}(b) is the unique minimal normal subgroup of $G$. Now $C_G(N)$ is normal in $G$, and $N\cap C_G(N)=1$ since $N$ is a nonabelian minimal normal subgroup. If $C_G(N)\ne1$ then $C_G(N)$ contains a minimal normal subgroup of $G$ distinct from $N$, which is a contradiction. Thus $C_G(N)=1$.
	
	Suppose first that $N$ is also a minimal normal subgroup of $\tilde{N}$. If $N$ is the only minimal normal subgroup of $\tilde{N}$, then part (a) holds. Suppose that this is not the case and that $M$ is a minimal normal subgroup of $\tilde{N}$ distinct from $N$. Then $M \leq C_{\tilde{N}}(N) \leq C_G(N)$, which is a contradiction since $C_G(N)=1$.  
	
	Suppose instead  that $N$ is not a minimal normal subgroup of $\tilde{N}$. Since $N$ is normal in $\tilde{N}$, this implies that $N$ has a subgroup $K$ which is minimal normal in $\tilde{N}$.
	Then since $K$ is also normal in $N$, it follows that $K = T^\ell$ where $T$ is the nonabelian simple group in the statement, and $\ell \geq 1$.  
	If $K$ were normal in $G$ then we would have $K=N$ (since $N$ is minimal normal in $G$) and hence $N$ would be minimal normal in $\tilde{N}$, a contradiction. Thus $K$ is not normal in $G$.
	
	 Recall  that $G/\tilde{N}$ is cyclic. Hence there exists an element $g \in G\backslash \tilde{N}$  such that $\langle \tilde{N}, g \rangle = G$. For $g$ with this property, we get $K^g \leq N^g=N \leq \tilde{N}$, and so $K^g$ is also a minimal normal subgroup of $\tilde{N}$ contained in $N$.
	
	Note that $g^r\in\tilde{N}$ and hence $g^r$ normalises $K$, and consider  $C = \{ K, K^g, \dots, K^{g^{r-1}}\}$, a set of minimal normal subgroups of $\tilde{N}$. The group $G = \langle \tilde{N}, g\rangle$ acts on $C$ by conjugation. Since $\tilde{N}$ normalises each $K^{g^{i}}\in C$, it acts trivially on $C$, and conjugation by $g$ cyclically permutes these subgroups transitively, with $g^r\in\tilde{N}$ acting trivially. It follows that $d:=|C|$ divides $r$.
	Note however that $d>1$, since otherwise $K$ would be normalised by $\langle \tilde{N}, g \rangle = G$, which is not the case.
	
	Now let $K_0$ = $\langle C \rangle =  \langle K , K^g, \dots , K^{g^{d-1}}\rangle = K \times \cdots \times K^{g^{d-1}} \cong T^{d\ell}$, and note that $K_0\leq N$ since each $K^{g^{i}}\leq N$. Since $G = \langle \tilde{N}, g \rangle$ is transitive on the $d\ell$ simple direct factors of $K_0$, it follows that $K_0$ is a minimal normal subgroup of $G$ and hence must be equal to $N$. Thus $N \cong T^{d\ell}$ with $\ell \geq 1$, $d>1$, $d$ dividing $r$, and so (b) holds.
	
	Finally, if $\tilde{N}$ has another minimal normal subgroup $M$ (distinct from the $K^{g^i}$) which is not contained in $N$, then $M$ commutes with each $K^{g^i}$ and so $M$ commutes with $N$. But this implies that $M \leq C_{\tilde{N}}(N) \leq C_G(N) = 1$. A contradiction, so all minimal normal subgroups of $\tilde{N}$ lie in $N$. 
\end{proof}
Our next goal is to provide an upper bound on the value of $k$ in Proposition \ref{OrientedTheorem}(b) when $N$ is nonabelian. As $N$ is the unique minimal normal subgroup of $G$, we may view $G$ as a subgroup of $\Aut(N) \cong \Aut(T) \wr \Sym(k)$, and identify $N \cong T^k$ with the group of inner automorphisms Inn$(N)$.

 Since $\tilde{N} = NG_\alpha$ with $N$ acting trivially on the set of $k$ direct factors, in both of the cases (a) and (b) of Lemma \ref{OrientedNonabelianSoc1}, there is a relationship between the value $k$ and the size of the vertex stabiliser $|G_\alpha|$.
We will use these ideas in the following proof which gives an upper bound for $k$. Note that this is a generalisation of an argument first used in $\cite{al2015finite}$ and then developed further in \cite{poznanovic2021four} to obtain analogous results for certain quasiprimitive and biquasiprimitive basic pairs.


\begin{Theorem}\label{OrientedNonabelianBound}
	Suppose that  $(\Gamma, G)$, $N$ and $k$ are as in Proposition $\ref{OrientedTheorem}(b)$ with $\Gamma_N \cong \mathbf{C}_r$ and $N$ nonabelian. Then either
	\begin{enumerate}[(a)]
		\item  $N$ is a minimal normal subgroup of $\tilde{N}$ and $k$ divides $2^r$; or
		\item  $N$ is not a minimal normal subgroup of $\tilde{N}$,  $k = d\ell$ for some divisor $d$ of $r$ with $d>1$, and $\ell$ divides $2^r$.
		
	\end{enumerate}
	In particular, $k$ divides $2^r$ in case (a), and $k$ divides $r2^r$ in case (b).
\end{Theorem}
\begin{proof}
The two cases (a) and (b) here directly correspond to the two cases (a) and (b) of Lemma \ref{OrientedNonabelianSoc1}. In particular, all we need to show is that in case (a) of Lemma \ref{OrientedNonabelianSoc1}, $k$ divides $2^r$ while in case (b) of Lemma \ref{OrientedNonabelianSoc1}, $\ell$ divides $2^r$.
	
	Consider $G$ as a subgroup of $\Aut(N) \cong \Aut(T) \wr \Sym(k)$, and identify $N \cong T^k$ with the group of inner automorphisms Inn$(N)$. Let $s\geq1$ be the largest integer such that $G$ acts transitively, and hence regularly (by Proposition \ref{localprops}), on the oriented $s$-arcs of $\Gamma$. Consider a vertex $\alpha:= \alpha_0$ of $\Gamma$ and let $(\alpha_0,\alpha_1,...,\alpha_s)$ be an oriented $s$-arc starting at $\alpha$. Suppose now that the pointwise stabiliser $ G_{{\alpha_0},...,{\alpha_{s-1}}} $ of order 2 is generated by some element $h_1$, that is, $G_{{\alpha_0},...,{\alpha_{s-1}}} = \langle h_1 \rangle \cong C_2$.
	
	Now let $g \in G$ be an automorphism of $\Gamma$ taking the oriented $s$-arc $(\alpha_0,\alpha_1,...,\alpha_s)$ to the oriented $s$-arc $(\alpha_1, \alpha_2,...,\alpha_s, \alpha_{s+1})$ where $\alpha_{s+1}$ is some out-neighbour of $\alpha_s$. Since $\Gamma_N$ is an oriented cycle, this implies that $g$ induces a rotation of the quotient graph $\Gamma_N$. Hence $\langle \tilde{N}, g \rangle$ is vertex-transitive and $g$ cyclically permutes the $N$-orbits on $V\Gamma$. Moreover $\tilde{N}g$ generates $G/\tilde{N} \cong C_r$ and $\langle \tilde N, g \rangle  = G$ since $G_\alpha \leq \tilde{N}$ by Lemma \ref{basics}.
	
	For $2 \leq i \leq s$, define $h_i := h^{g^{-1}}_{i-1}$. It is clear that for each $i \leq s$ we have $$ G_{{\alpha_0},...,{\alpha_{s-i}}} = \langle h_1, ..., h_i\rangle.$$
	Thus $|G_\alpha| = 2^s$ and $G_\alpha = \langle h_1, \dots, h_s\rangle$ where each $h_i$ is an involution. 
	
	We may write the automorphisms $h_1, g \in G$ as elements of $\Aut(N) \cong \Aut(T) \wr \Sym(k)$, so that $h_1 = f\varphi$ and $g = f'\sigma$ where $f, f' \in \Aut(T)^k$ and $\varphi, \sigma \in \Sym(k)$. (In case (b), $\varphi, \sigma \in \Sym(\ell) \wr \Sym(d)$ with $\varphi \in \Sym(\ell)^d$.) In either case, $h_1^2 =1$ implies that $\varphi^2 =1$.
	
	Now let $\pi$ denote the projection map $\pi: \Aut(N) \rightarrow \Sym(k)$
	, so that $(h_1)\pi = \varphi$ and $(g)\pi = \sigma$, and let $P := (\tilde{N})\pi = (NG_{\alpha})\pi = (G_{\alpha})\pi$. Note that $P$ is a 2-group since $G_{\alpha}$ is a 2-group, and further  $$P = (G_{\alpha})\pi = \langle h_1,h_2,..., h_s\rangle\pi = \langle \varphi, \varphi^{\sigma^{-1}},..., \varphi^{\sigma^{-(s-1)}}\rangle.$$
	Moreover, in case (a), $P$ is transitive implying that $k$ divides $|P|$, while in case (b), $P$ has $d$ orbits of length $\ell$, so $\ell$ divides $|P|$. Note that $|P|$  divides  $|G_\alpha| = 2^s$ since $P$ is a quotient of $G_\alpha$. Hence $k$ divides $2^s$ in case (a) and $\ell$ divides $2^s$ in case (b). If $s \leq r$ then  $k$ divides $2^r$ in case (a) and $\ell$ divides $2^r$ in case (b) and the result holds. We will henceforth assume that $s>r$.
	
	First, notice that $\varphi$ is not contained in any proper $\sigma$-invariant subgroup of $P$. For suppose that it were, and take a proper $\sigma$-invariant subgroup $\overline{P}$ of $P$ containing $\varphi$. Since $\overline{P}$ is $\sigma$-invariant it follows that $\varphi^{\sigma^{-i}} \in \overline{P}$ for all $i \in \mathbb{Z}$, implying that $P\leq \overline{P}$ and hence that $P = \overline{P}$, a contradiction.
	
	Now consider the index $i_0 := |(G)\pi:P|  \geq 1$. Note that $i_0$ divides $|G:\tilde{N}| = r$ and $P$ is a normal subgroup of $(G)\pi$. Moreover in case (b), $(G)\pi$ is transitive while $P$ has $d$ orbits of length $\ell$ and these $P$-orbits form a $(G)\pi$-invariant partition of $\{1,\dots ,k\}$. Hence in case (b), the group $(G)\pi / P $ of order $i_0$ acts transitively on the set of these $d$ $P$-orbits and so $d$ divides $i_0$. Thus we have $d$ $ | $ $i_0 $ $|$ $ r$.
	
	Suppose first that $i_0 = 1$ and hence that $P = (G)\pi$. Note that this cannot occur in case (b) as this implies that $d = 1$. Note also that  $\sigma\in (G)\pi=P$ in case (b). Suppose that $\langle \varphi \rangle$ is a proper subgroup of $P$ and let $M$ be a maximal subgroup of $P$ containing $\langle \varphi \rangle$. Since $P$ is a 2-group it follows that $M$ is normal in $P$ and hence must be $\sigma$-invariant. Recall however that $\varphi$ cannot be contained in any proper $\sigma$-invariant subgroup of $P$. It follows that $P = \langle \varphi \rangle$ with order at most 2. Therefore  $k \leq 2$ and hence $k$ divides $2^r$.
	
	Suppose on the other hand that $1 < i_0 \leq r$ where  $d$ $ | $ $i_0 $ $|$ $ r$. In this case, since $\langle P, \sigma\rangle = (G)\pi$,  we have $\sigma\in (G)\pi \backslash P$, and moreover $(G)\pi/P \cong C_{i_0}$ with  $\sigma^{i_0}\in P$.
	
	Now let $L := \Phi(P)$, the Frattini subgroup of $P$ and note that $P/L$ is elementary abelian. Then $L$ is $\sigma$-invariant since $\sigma$ normalises $P$, so $L$ does not contain $\varphi$. Let $J := \langle L, \varphi \rangle$ and note that conjugation by $\sigma^{-1}$ in $(G)\pi$ maps $J$ to $J^{\sigma^{-1}} = \langle L, \varphi^{\sigma^{-1}} \rangle$. Since $P/L$ is elementary abelian, it follows that $J$ is normal in $P$. In particular, since $\sigma^{i_0} \in P$, conjugation by $\sigma^{-i_0}$ fixes $J$.
	
	Now consider the set $S = \{J, J^{\sigma^{-1}},\dots, J^{\sigma^{-(i_{0}-1)}}\}$ of subgroups of $P$. Conjugation  by $\sigma^{-1}$ cyclically permutes the elements of the set $S$ and hence $|S| = c$ where $c$ divides $i_0$.  Furthermore, each generator $\varphi^{\sigma^{j}}$ of $P$ (with $-(s-1)\leq j \leq 0$), lies in some $J ^{\sigma^{-i}}$ with $0\leq i < i_0$. In particular $\varphi$ lies in some $J ^{\sigma^{-i}}$. Thus $\langle S \rangle$ is a $\sigma$-invariant subgroup of $P$ containing $\varphi$, and hence $P = \langle S \rangle = \langle L , \varphi, \varphi^{\sigma^{-1}}, \dots \varphi^{\sigma^{-c}} \rangle = \langle \varphi, \varphi^{\sigma^{-1}}, \dots \varphi^{\sigma^{-c}} \rangle$ where the last equality holds since $L$ is the Frattini subgroup.
	
	Hence  $P = \langle \varphi,\dots, \varphi^{\sigma^{-c}} \rangle = \langle h_1,\dots, h_c \rangle \pi$, and since $\langle h_1,\dots, h_c \rangle$ has order $2^c$ it follows that $|P|$ divides $2^c$. Hence in case (a), $P = (\tilde{N})\pi$ is transitive in $\Sym(k)$ and so $k$ divides $2^c$, while in  case (b), $P$ has $d$ orbits of length $\ell$, so $\ell$ divides $2^c$.  In either case, $c \leq i_0 \leq r$ and so either $k$ divides $2^r$ in case (a), or $\ell$ divides $2^r$ in case (b).
\end{proof}

\subsection{Oriented-Cycle Type: abelian socle}
Suppose now that $(\Gamma, G)\in \OG(4)$ is basic of oriented-cycle type as described by Proposition \ref{OrientedTheorem}(b) with $N \cong T^k$ for an abelian simple group $T\ncong\mathbb{Z}_2$ and $k \geq 1$. 
 We will now obtain the bound on $k$ for pairs of this type as given in Theorem \ref{CycleMainTheorem}.

\begin{Theorem}\label{OrientedAbelianBound}
	Suppose that  $(\Gamma, G)$, $N, T$ and $k$ are as in Proposition $\ref{OrientedTheorem}(b)$ with $\Gamma_N \cong \mathbf{C}_r$ for $r \geq 3$ and $N$ abelian. Then $k\leq r$. 
\end{Theorem}
\begin{proof}
    Consider some vertex $\alpha$ contained in an $N$-orbit $\Delta$. First note that since $N$ is a $p$-group, with $p$ odd, and $G_\alpha$ is a 2-group, it follows that $N_\alpha = 1$ and $N$ is semiregular with $r$ orbits of length $p^k$ on $V\Gamma$. Hence $\tilde{N} = N \rtimes G_\alpha$ by Lemma \ref{OrientedFaithful}, and $G_\alpha \cong \mathbb{Z}_2^s$ with $s \leq r$ by Lemma \ref{Abelianstabiliser}. 

     Since each $N$-orbit contains $p^k$ vertices and since $G_\alpha$ is a 2-group, it follows that $G_\alpha$ fixes at least one vertex in each $N$-orbit. Now we know that the normaliser $N_G(G_\alpha)$ is transitive on the fixed-point set of $G_\alpha$ (see for instance \cite[Corollary 2.24]{praeger2018permutation}) and so  $N_G(G_\alpha)$ acts transitively on the set of $N$-orbits. Hence $N_G(G_\alpha)$ is transitive on the quotient graph $\Gamma_N$.
	
	We may therefore take some $\sigma\in N_G(G_\alpha)$ such that $\sigma$ acts transitively as a rotation of $\Gamma_N$. Now $H = \langle N, G_\alpha, \sigma \rangle \leq G$ is transitive on $V\Gamma$, and since $H$ contains $G_\alpha$, it follows that $H = G$. We thus have $G = \langle N, G_\alpha, \sigma \rangle  = \langle \tilde{N} ,\sigma \rangle$ with $\sigma \in N_G(G_\alpha)$ and $\sigma^r \in \tilde{N}$. 
	
	Let $\{\varphi_1, \dots, \varphi_s\}$ be a set of generators of $G_\alpha$ so that $G_\alpha = \langle \varphi_1, \dots, \varphi_s : \varphi_i^2 = 1, [\varphi_i, \varphi_j] =1, i \neq j \rangle \cong \mathbb{Z}_2^s.$
	We will now view $N$ as a $k$-dimensional vector space $V$ over $\mathbb{Z}_p$. Then the conjugation action of $G$ on $N$ gives a group homomorphism $\rho : G \rightarrow \GL_k(p)$, where $\rho(G)$ is irreducible (since $N$ is a minimal normal subgroup of $G$). Also, $N \leq $ Ker$\rho$, so $\rho(G) = \rho(\langle G_\alpha, \sigma \rangle)$. From now on we will suppress the symbol $\rho$ in the notation.
	

	Take an arbitrary nonzero vector $v_0 \in V= \mathbb{Z}_p^k$ and consider the vector $v_0 + v_0^{\varphi_1}$. If $v_0 + v_0^{\varphi_1} = 0$ then $v_0^{\varphi_1} = -v_0$ so $v_0$ is an eigenvector of $\varphi_1$ with eigenvalue $-1$. On the other hand, if $v_0 + v_0^{\varphi_1} \neq 0$ then $v_0 + v_0^{\varphi_1}$ is an eigenvector of $\varphi_1$ with eigenvalue 1. Hence we may define a vector $v_1 \in V$ as
	$$ v_1 := \begin{cases}
	&v_0, \hbox{ \hspace{13.3mm}      if  }  v_0 + v_0^{\varphi_1} = 0 \\
	&v_0 + v_0^{\varphi_1}, \hbox{\hspace{5mm}  if  }v_0 + v_0^{\varphi_1}\neq 0,
	\end{cases}$$
	and $v_1$ will be an eigenvector of $\varphi_1$ with eigenvalue $\pm 1$. 
	
	Now for $1 < i \leq s$ we  define a vector $v_i$ as follows 
	$$ v_i := \begin{cases}
	&v_{i-1}, \hbox{  \hspace{14.3mm}    if  }  v_{i-1} + v_{i-1}^{\varphi_i} = 0 \\
	& v_{i-1} + v_{i-1}^{\varphi_i}, \hbox{\hspace{5mm}  if  } v_{i-1} + v_{i-1}^{\varphi_i} \neq 0.
	\end{cases}$$
	We claim that $v_i$ is a common eigenvector for all elements in the set $\{\varphi_1, \dots, \varphi_i\}$ having an eigenvalue of $\pm 1$ for each $\varphi_j$ where $1 \leq j \leq i$. We prove this by induction on $i$.
	
	First, we have already shown that $v_1$ is an eigenvector of $\varphi_1$ with eigenvalue $\pm 1$. Now assume that $i \geq 2$ and that $v_{i-1}$ is an eigenvector for all elements in $\{\varphi_1, \dots, \varphi_{i-1}\}$ with  eigenvalues $\pm 1$ for each $\varphi_j$ where $1 \leq j < i$ . First, if  $v_{i-1} + v_{i-1}^{\varphi_i} = 0$, then $v_{i-1}^{\varphi_i} = -v_{i-1}$, and in this case $v_i = v_{i-1}$ is an eigenvector of $\varphi_i$ with eigenvalue $-1$, and by induction $v_i$ is an eigenvector for each $\varphi_j$ (where $j<i$) with eigenvalue $\pm1$. So the inductive hypothesis is proved in this case. Suppose now that  $v_{i-1} + v_{i-1}^{\varphi_i} \neq 0$. Then $v_i = v_{i-1} + v_{i-1}^{\varphi_i}$ and $v_i$ is clearly an eigenvector of $\varphi_i$ with eigenvalue $1$. By induction, for any positive integer $ j < i$, we know that $v_{i-1}^{\varphi_j} = \varepsilon v_{i-1}$ (for some $\varepsilon = \pm1$)  and so  
	$$
    v_i^{\varphi_j} = (v_{i-1} + v_{i-1}^{\varphi_i})^{\varphi_j} = v_{i-1}^{\varphi_j} + v_{i-1}^{\varphi_i\varphi_j} = v_{i-1}^{\varphi_j} + v_{i-1}^{\varphi_j\varphi_i} = \varepsilon v_{i-1} + \varepsilon v_{i-1}^{\varphi_i} = \varepsilon( v_{i-1} + v_{i-1}^{\varphi_{i}}) = \varepsilon v_i,
    $$
	which implies that $v_i$ is also an eigenvector of each of $\varphi_1, \dots ,\varphi_{i-1}$ having an eigenvalue of $\pm 1$ for each.  This  proves the result. Hence the vector $w:= v_s$ defined as above will be a common eigenvector for  $\{\varphi_1, \dots , \varphi_s\}$ and so $G_\alpha$ preserves the 1-dimensional subspace of $V$ spanned by $w$. Also since $\sigma \in N_G(G_\alpha),$ we know that $\langle w^{\sigma^i}\rangle$ is $G_\alpha$-invariant, and in fact the set $\{ w^{\sigma^i}, -w^{\sigma^i}\}$ is $G_\alpha$-invariant, for each $i$.
	
	Now consider the set $S = \{ w, w^\sigma, \dots,  w^{\sigma^{r-1}}\}$. We claim that the span of this set is preserved under conjugation by  $G$. Since $N$ centralises $S$ and $G_\alpha$ preserves each $\langle w^{\sigma^i} \rangle$, this span is preserved under conjugation by $\tilde{N} = N\rtimes G_\alpha$.
	Moreover, it is easy to see that conjugation by $\sigma$ simply maps each $w^{\sigma^i} \in S$ to $w^{\sigma^{i+1}}$ for $i \in [0,r-2]$. Furthermore,  $\sigma^r \in \tilde{N} = N \rtimes G_\alpha$, so we have $\sigma^r = nh$ for some $n\in N$ and $h\in G_\alpha$. It follows that $w^{\sigma^r} = w^{h} = \pm w$. Thus $\sigma$ maps  $w^{\sigma^{r-1}}$ to $w^{\sigma^r} = \pm w$ which again lies in the span of $S$. Hence $\sigma$ also preserves the subspace spanned by $S$.

	
	Thus span$\{w, w^\sigma, \dots, w^{\sigma^{r-1}}\}$ is preserved under conjugation by $\langle \tilde{N}, \sigma \rangle = G$. So by the minimality of $N$ we have $V = $ span$\{w, w^\sigma, \dots, w^{\sigma^{r-1}}\}$, and so  dim$(V) = k \leq r$. 
\end{proof}

We have now shown that if $(\Gamma, G) \in \OG(4)$ is basic of oriented-cycle type then one of the subcases of  Case 1 of Theorem \ref{CycleMainTheorem} holds.
\section{Unoriented-Cycle Type}\label{secUnori}
We now turn our attention to the basic pairs of unoriented-cycle type.  Recall that a basic pair $(\Gamma, G) \in \OG(4)$ with cyclic normal quotients is of unoriented-cycle type if it does not have independent cyclic normal quotients and all of its cyclic normal quotients are $G$-unoriented. All such pairs will have a vertex stabiliser of order 2 by Lemma \ref{basics}. We will now show that for such a pair $(\Gamma, G)$, the group $G$ will always have a
unique minimal normal subgroup.

\begin{Proposition}\label{unorientedUnique}
	Suppose $(\Gamma, G)\in \OG(4)$ is basic of unoriented-cycle type. Let $r\geq 3$ be the largest number of orbits of any nontrivial normal subgroup of $G$.
	Then $G$ has a unique minimal normal subgroup $N \cong T^k$ for some simple group $T$ and $k\geq 1$, and  $\Gamma_N \cong \mathbf{C}_r$. 
\end{Proposition}
\begin{proof}
	Let $N$ be a normal subgroup of $G$ with $r$ orbits. Then $\Gamma_N \cong \mathbf{C}_r$ is a $G$-unoriented cycle by assumption and $N= \tilde{N}$ is semiregular on $V\Gamma$ by Lemma \ref{basics}. Now any nontrivial normal subgroup $N_0$ of $G$ contained in $N$ must have $r$ orbits by the  maximality of $r$. Hence $N_0$ is transitive on each $N$-orbit, and $N = N_0$ since $N$ is semiregular. Thus $N$ is a minimal normal subgroup of $G$.
	
	Now suppose that $M$ is an arbitrary minimal normal subgroup of $G$. Then by Lemma \ref{3orbs}, $\Gamma_M \cong \mathbf{C}_s$ with $s \geq 3$, and so by the maximality of $r$, we have $3 \leq s \leq r$. Moreover, $\Gamma_M$ is $G$-unoriented and $\tilde{M} = M$ by Lemma \ref{basics}. Since $(\Gamma, G)$ is of unoriented-cycle type, it does not have independent cyclic quotients. Thus $\Gamma_{\tilde{N}\cap \tilde{M}}$ is a cyclic normal quotient, and in particular $\tilde{N}\cap \tilde{M} = N\cap M \neq 1$. Hence $N = M$, and $N$ is the unique minimal normal subgroup of $G$.
\end{proof}

We will now focus on the parameters $k$ and $r$ of Proposition \ref{unorientedUnique}. As we did in the oriented case, we will bound $k$ by a function of $r$. Once again we obtain two different bounds depending on whether $N$ is abelian or nonabelian. We will deal with the nonabelian case (the simpler of the two) first. 
\subsection{Unoriented-Cycle Type: nonabelian socle}
The following result gives a bound on the number of simple direct factors of a nonabelian minimal normal subgroup $N$ of a group $G$ where $(\Gamma, G)\in\OG(4)$ is basic of unoriented-cycle type. 

\begin{Theorem}\label{UnorientedNonabelianBound}
	Let $(\Gamma, G), N, k$, and $r$ be as in Proposition $\ref{unorientedUnique}$ and suppose that $N$ is nonabelian. Then  $k$ divides $2r$, and in particular, $k\leq 2r$. 
\end{Theorem}

\begin{proof}
		Let $(\Gamma, G)$ be as in the statement of the theorem, and suppose that the unique minimal normal subgroup $N$ of $G$ is nonabelian. So $N=T^k$ with $T$ a nonabelian simple group and $k\geq1$. Since $N = \soc(G)$ and $C_G(N) = 1$, we can  view $G$ as a subgroup of $\Aut(N) = \Aut(T) \wr \Sym(k)$, identifying $N$ with the subgroup Inn$(N) \leq \Aut(N)$ of inner automorphisms. Now $D_r \cong G/N \lesssim $ Out$(N)$, and $G/N$ acts transitively by conjugation on the $k$ simple direct factors of $N$. Hence $k$ divides $|D_r|=2r$ and $k \leq  2r$. 
\end{proof}

\subsection{Unoriented-Cycle Type: abelian socle}

Suppose now that $(\Gamma, G)\in \OG(4)$ is basic of unoriented-cycle type, with a unique minimal normal subgroup $N = T^k$ where $T= \mathbb{Z}_p$ of some prime $p$, and $k \geq 1$. Suppose that $\Gamma_N \cong \mathbf{C}_r$ for some $r\geq 3$. In this case we have $G = N \langle \sigma, \varphi\rangle$ where $\sigma ^r  \in N$ and $\langle \varphi \rangle = G_\alpha \cong \mathbb{Z}_2$. Moreover, $G/N \cong D_r$ and hence $\sigma^\varphi = \sigma^{-1}$ modulo $N$, so 
$\sigma^\varphi \in N\sigma^{-1}$. Since $N$ is minimal normal in $G$, we know that $G$, and in particular, $G/N\cong D_r$ acts irreducibly on $N= \mathbb{Z}_p^k$. Our goal is to determine the degrees of the (possibly unfaithful) irreducible representations of $D_r$ on the vector space $\mathbb{F}_p^k$ and hence determine the possible values of integers $r$, $p$, and $k$ for which such pairs $(\Gamma, G)$ exist. 

To obtain our bound we will first need to develop some theory regarding group representations. We will provide our bound on $k$ in Lemma \ref{UnorAffineBound}. 

We begin with a standard result relating irreducible elements of $\GL(\mathbb{F}_p^k)$ to irreducible polynomials in $\mathbb{F}_p[X]$ of degree $k$. For details on the companion matrix construction and the rational canonical form used in the following lemma we refer the reader to \cite[Chapter 11]{hartley1970rings}. 

\begin{Lemma}\label{irredcorr} Let $V = \mathbb{F}_p^k$ and let $\sigma\in \GL(V)$. If $\langle \sigma \rangle$ acts irreducibly on $V$ then $\sigma$ is similar to the companion matrix of some irreducible non-constant monic polynomial $f \in \mathbb{F}_p[X]$ of degree $k$. 

	Conversely, the companion matrix of any irreducible non-constant monic polynomial $f \in \mathbb{F}_p[X]$ of degree $k$ will be an element $\sigma \in \GL(V)$ such that $\langle \sigma \rangle$ acts irreducibly on $V$.
\end{Lemma}

\begin{proof}
	Suppose first that $\langle \sigma \rangle$ acts irreducibly on $V$. In this case, the minimal polynomial $m_\sigma(X)$ of $\sigma$ is irreducible over $\mathbb{F}_p$ since the existence of a proper divisor $f \in \mathbb{F}_p[X]$ of $m_\sigma(X)$ would imply the existence of a proper nontrivial $\sigma$-invariant subspace of $V$, namely $\ker(f(\sigma))$, which would contradict the irreducibility of $\sigma$. 
	By \cite[Theorem 11.17]{hartley1970rings}, $\sigma$ is similar to some matrix $Q_\sigma = C(f_1) \oplus \cdots \oplus C(f_d)$ where each $C(f_i)$ is the companion matrix of some non-constant monic polynomial $f_i \in \mathbb{F}_p[X]$ (this is the rational canonical form). Since $\sigma$ acts irreducibly on $V$, it follows that $d = 1$, and so $\sigma$ is similar to the companion matrix $C(f_1)$ of some polynomial $f_1$. In particular the characteristic and minimal polynomials of $C(f_1)$ (and hence also of $\sigma$), are both equal to $f_1$. Thus $m_\sigma(X)= f_1$ is irreducible of degree $k$, and $\sigma$ is similar to the companion matrix of $f_1$. 
	
	
	Conversely, consider an irreducible non-constant monic polynomial $f \in \mathbb{F}_p[X]$ of degree $k$ and let $\sigma := C(f)$ denote the companion matrix of $f$. Then the minimal polynomial $m_\sigma(X)$ of $\sigma$ is equal to $f$. To see that $\langle \sigma \rangle $ acts irreducibly on $V$ we simply show that for any non-zero $v\in V$, the subspace $W:= \mathbb{F}_p[\sigma](v) = \{z(\sigma)(v) : z(X) \in \mathbb{F}_p[X]\}$ of $V$ is equal to $V$.
	
	%
	To this end, take an arbitrary non-zero $v\in V$, and let $m(X) \in \mathbb{F}_p[X]$  be a polynomial of minimal degree such that $m(\sigma)(v) = 0$. Then writing $m_\sigma(X) = m(X)q(X) + r(X)$ where $q(X), r(X) \in \mathbb{F}_p[X]$ and deg $r(X) <$ deg $m(X)$, we see that $0 = m_\sigma(\sigma)(v) = m(\sigma)q(\sigma)(v) +r(\sigma)(v)$, and so $r(X) = 0$ by minimality of the degree of $m(X)$. Hence $m(X) | m_\sigma(X)$, which implies that $m(X) = m_\sigma(X)$ since $m_\sigma(X)$ is irreducible.
	Therefore if we define a linear map $\psi: \mathbb{F}_p[X] \rightarrow W$ by $z(X) \mapsto z(\sigma)(v)$, we see that $\ker \psi = \{z(X) \in \mathbb{F}_p[X] : z(\sigma)(v) = 0\} = \{z(X) \in \mathbb{F}_p[X] : m_\sigma(X)$ $|$  $z(X)\} = \langle m_\sigma(X) \rangle$, where $\langle m_\sigma(X) \rangle$ is the ideal generated by $m_\sigma(X)$. Hence, as vector spaces,  $\mathbb{F}_p[X] /\langle m_\sigma(X)\rangle \cong W$ and since the dimension of $\mathbb{F}_p[X] /\langle m_\sigma(X)\rangle$ is $k$, it follows that $W$ has dimension $k$ and hence must be equal to $V$. Hence $\langle \sigma \rangle$ acts irreducibly on $V$.
\end{proof}

Our aim is to determine the possible degrees of the irreducible representations of the dihedral group $D_r$. As an intermediate step, we will first determine the irreducible representations of the cyclic group $C_r$. In particular we will show that $C_r$ has a unique irreducible representation of degree $\phi(r)$ where $\phi(r)$ is the Euler totient function. This will be a crucial result towards  bounding  $k$.  

The field $F := \mathbb{F}_{p^k}$ can be thought of as a vector space of dimension $k$ over $\mathbb{F}_p$. Multiplication by a primitive root $\alpha \in \mathbb{F}_{p^k}^*$ then gives an element $\sigma \in \GL_k(p)$ of order $p^k -1$. This element $\sigma\in \GL_k(p)$, and sometimes also the subgroup it generates, is known as a \textit{Singer cycle}. The next result can be found in Huppert's book \cite[Satz II.7.3 p.187]{huppert1967I}.

\begin{Lemma}\label{Huppert}
	Let $G = \GL_k(p)$ and let $\sigma \in G$ act irreducibly on the vector space $\mathbb{F}_p^k$. Then the centraliser $C_G(\sigma)$ is a Singer cycle. 
\end{Lemma}
Note that Lemma~\ref{Huppert} immediately implies that the order of an irreducible element $\sigma\in \GL_k(p)$ divides $p^k-1$. In what follows we will also make use of the following elementary result. A proof can be found in \cite[p.36, Example 95]{santos2007number}.

\begin{Lemma}\label{gcdlemma}
    Let $a, m$ and $n$ be positive integers with $a\neq 1$. Then $$\gcd(a^{m}-1, a^n-1) = a^{\gcd(m,n)}-1.$$
\end{Lemma}
We will now use these facts to prove the following.

\begin{Lemma}\label{cyclicrep}
	Let $V = \mathbb{F}_p^k$ and let $(\rho, V)$ be a faithful irreducible representation of the cyclic group $C_{r}$ for $r\geq 2$. Let $\phi$ denote the Euler totient function. Then the following hold: 
	\begin{enumerate}[(i)]
		\item $k$ is the smallest integer such that $r$ $ |$ $ p^k -1$, in particular $\gcd(p,r) = 1$,
		\item $k $ divides $\phi(r)$,
		\item there is a unique irreducible $C_{r}$-representation of degree $\phi(r)$. 
	\end{enumerate}
\end{Lemma}
\begin{proof}
	Suppose that $\rho(C_r) = \langle \sigma \rangle \leq \GL(V)$, where $\sigma$ is an irreducible element of order $r$ and degree $k$.
     We may identify $V$ with the field $\mathbb{F}_{p^k}$, and apply Lemma \ref{Huppert} so that $\langle \sigma \rangle$ acts as multiplication by some element $\alpha \in \mathbb{F}_{p^k}^*$, and  $r$ divides $p^k-1$. In particular,  gcd$(p,r)=1$. 
	
	By Lemma \ref{gcdlemma}, for integers $m$ and $ n$ we have $\gcd(p^m-1, p^n-1) = p^u -1 $, where $u = \gcd(m,n)$. Thus if  $i$ is the smallest integer such that $r$ divides $p^i-1$, then since $r$ divides $p^k-1$, it follows that $r$ divides $p^u-1$ where $u = \gcd(i,k)$. Minimality of $i$ then implies that $i = u$ and hence that $i$ divides $k$. Thus $p^i-1$ divides $p^k-1$. 
	But if $r$ divides $p^i-1$ with $i <k$, then $\alpha$ lies in the proper subfield $\mathbb{F}_{p^i}$ of $\mathbb{F}_{p^k}$ contradicting the fact that $C_r$ is irreducible. Hence  $r$ does not divide $p^i-1$ for any $i<k$ and (i) holds.

	
	For (ii), note that since $p$ and $r$ are coprime, we have $p^{\phi(r)} \equiv 1$ mod $r$. Now $r$ divides $p^u-1$ where $u =$ gcd$(k,\phi(r))$. By (i), we then get that  $u = k$ and hence $k$ divides $\phi(r)$.

	For (iii), consider again an irreducible element  $\sigma \in  \GL(V)$ of order $r$ and degree $k$. By the first part of Lemma \ref{irredcorr}, $\sigma$ is similar to the companion matrix of some irreducible polynomial $m_\sigma(X)\in \mathbb{F}_p[X]$ of degree $k$, and moreover $m_\sigma(X)$ is the minimal polynomial of $\sigma$. Furthermore since $\sigma^r= 1$, it follows that $m_\sigma(X)$ divides $X^r-1$. 
	
	Thus the number of faithful irreducible representations of $C_r$ of degree $k$ is equal to the number of irreducible monic polynomials  $f(X) \in \mathbb{F}_p[X]$ of degree $k$
	 which are divisors of $X^{r}-1$. 
	 By  \cite[Theorem 3.5]{lidl1997finite}, the number of such irreducible polynomials of degree $\phi(r)$ is equal to $\phi(r)/\phi(r) = 1$. It follows that there is a unique faithful irreducible representation of $C_r$ of degree $\phi(r)$.  
\end{proof}

We  now determine the degrees of the faithful irreducible representations of the dihedral group $D_r$ on $V = \mathbb{F}_p^k$. As before, we identify $V$ with $\mathbb{F}_{p^k}$. Then an irreducible element $\sigma \in \GL(V)$ corresponds to multiplication by some element $\alpha \in \mathbb{F}_{p^k}^*$,  and since the order of $\sigma$, and hence the order of $\alpha$, does not divide $p^i-1$ for any $i < k$ (by Lemma \ref{cyclicrep}(i)), it follows that $\alpha$ does not lie in any proper subfield of $\mathbb{F}_{p^k}$. 

\begin{Lemma}\label{dihedralrep}
	Let $V = \mathbb{F}_p^k$ and let $(\rho, V)$ be a faithful irreducible representation of the dihedral group $D_{r}$ for $r\geq 3$.
	Then one of the following holds
	\begin{enumerate}[(i)]
		\item $k$ is the smallest integer such that $r $ $ |$ $p^k-1$; or
		\item $k = 2\ell$ where $\ell$ is the smallest integer such that $r $ $ |$ $p^\ell-1$.
	\end{enumerate}
	Moreover, in case (ii) we have $\ell \leq \phi(r) /2$.
\end{Lemma}
\begin{proof}
	Let $G := D_r = \langle \sigma,  \varphi : \sigma^r = \varphi^2 = 1,$ $  \sigma^\varphi = \sigma^{-1}\rangle$ and let $H$ be the index 2 cyclic subgroup $H := \langle \sigma \rangle \leq D_{r}$. Let $W$ be an $H$-invariant subspace of $V$ of maximal dimension on which $H$ acts irreducibly. If $W=V$ then $\rho$ restricts to an irreducible representation of $H$ and by Lemma \ref{cyclicrep}, case (i) holds. 
	
	Suppose instead that $W$ is a proper subspace of $V$. We will show that (ii) holds. Again, by Lemma \ref{cyclicrep} the dimension of $W$ is $\ell$, where $\ell$ is the smallest integer such that $r $ divides $p^\ell-1$. In this case taking $\varphi \in G \backslash H$ we see that $W + W^\varphi$ is a nontrivial $G$-invariant subspace of $V$ and hence $W + W^\varphi=V$ as $D_r$ is irreducible. Moreover,  $W \neq W^\varphi$ since $W\neq V$. Now $W \cap W^\varphi$ is $G$-invariant since it is $H$-invariant by the definition of $W$, and it is $\varphi$-invariant since $\varphi^2 =1$. Thus $W \cap W^\varphi$ must be trivial, so $V = W \oplus W^\varphi$ and $k = 2\ell$ as claimed. In this case both $W$ and $W^\varphi$ are irreducible $\mathbb{F}_p$ $H$-modules of dimension $\ell$. 
	
	We now prove the final claim. By Lemma~\ref{cyclicrep}(ii), $\ell$ divides $\phi(r)$, so the final claim holds if $\ell < \phi(r)$. Suppose for a contradiction that $\ell=\phi(r)$. Then, by Lemma~\ref{cyclicrep}(iii),  $W$ and $W^\varphi$ are isomorphic $\mathbb{F}_{p^\ell}$ $H$-modules. Also, since $r\geq3$, it follows that $\ell=\phi(r)>1$. We showed in the previous paragraph that $\varphi \in G$ interchanges the two subspaces $W$ and $W^\varphi$, and so it follows that $G$ is an imprimitive linear group $G \leq \GL(W) \wr S_2 \leq \GL(V)$. Since $H$ is irreducible, by Lemma \ref{Huppert} we can identify $H$ with a subgroup of $\GL_1(p^\ell) \cong \mathbb{F}_{p^\ell}^*$, and identify $W$ with $\mathbb{F}_{p^\ell}$, such that $H$ acts on $W$ by multiplication.
	
	With this identification, we may write any element $v \in V$ as $v = (w_1, w_2^\varphi)$ where $w_1,w_2 \in W = \mathbb{F}_{p^\ell}$, so that $W = \{(w, 0): w \in W\}$ and $W^\varphi = \{(0,w^\varphi): w \in W\}$. The action of $\varphi \in G$ interchanges $W$ and $W^\varphi$ via  $(w,0) \longleftrightarrow (0, w^\varphi)$, and the element $\sigma$ maps $(w,0)$ to $(w^\sigma, 0)$ and $(0,w^\varphi)$ to $(0,w^{\varphi\sigma}) = (0, (w^{\sigma^{-1}})^\varphi)$, 
	noting that $\varphi\sigma\varphi = \sigma^{-1}$. 
	The restrictions $\sigma|_W$ and $\sigma|_{W^{\varphi}}$ are both elements of $\GL_1(p^\ell)$, so $\sigma \in \langle \sigma|_W \rangle \times \langle \sigma|_{W^{\varphi}}\rangle = \GL_1(p^\ell)^2$. Therefore we have $G = \langle \sigma, \varphi \rangle  \leq (\langle \sigma|_W \rangle \times \langle \sigma|_{W^{\varphi}}\rangle)\cdot \langle \varphi \rangle \leq \Gamma$L$_1(p^\ell) \wr S_2$ where $\Gamma$L$_1(p^\ell)$ is the full group preserving the $\mathbb{F}_{p^\ell}$ structure on $W$, namely $\Gamma$L$_1(p^\ell) = \mathbb{F}_{p^\ell}^* \rtimes \Aut(\mathbb{F}_{p^\ell})$. 
	We next identify  $\sigma$ and $\varphi$  as elements of $\Gamma$L$_1(p^\ell) \wr S_2$. Now $\sigma|_W$ and $\sigma|_{W^{\varphi}}$ act  by multiplication by some scalars $\alpha$ and $\alpha'$ of $\mathbb{F}_{p^\ell}^*$, respectively, and as $W$ and $W^\varphi$ are isomorphic $\mathbb{F}_{p^\ell}$ $H$-modules, we must have $\alpha = \alpha'$, so that  $\sigma = (\alpha, \alpha) \in \Gamma$L$_1(p^\ell)^2$. Also, the element $\varphi \in G$ may be written as $\varphi = (x, y)(12)$ with $x, y \in \Gamma$L$_1(p^\ell)$ and $(12) \in S_2$, and since $1 = \varphi^2 = ((x, y)(12))^2 = (x, y)(12)(x, y)(12) = (xy,yx)$, it follows that $y = x^{-1}$ so that $\varphi = (x,x^{-1})(12)$, for some $x \in \Gamma$L$_1(p^\ell)$.
	
    Note that in this representation of $\sigma$ we are identifying the element $\alpha \in \mathbb{F}_{p^\ell}^*$ with the operation of `multiplication by $\alpha$' in $W$. Note also that  by Lemma \ref{Huppert} and the discussion immediately preceding this theorem, $\alpha$ lies in no proper subfield of $\mathbb{F}_{p^\ell}$. Since $\varphi\sigma \varphi = \sigma^{-1}$, we have 
    $$
    (\alpha^{-1},\alpha^{-1}) = \sigma^{-1} = \varphi\sigma\varphi = (x,x^{-1})(12)(\alpha, \alpha )(x,x^{-1})(12) = (x\alpha x^{-1}, x^{-1}\alpha x),
    $$ 
    and hence $\alpha = x^{-1}\alpha^{-1} x$. That is, $x \in \Gamma$L$_1(p^\ell)$ inverts $\alpha \in \mathbb{F}_{p^\ell}^*$.
	Since $x \in \Gamma$L$_1(p^\ell)$, we may write  $x = x_0\xi$ where $x_0 \in \mathbb{F}_{p^\ell}^*$ and $\xi \in \Aut(\mathbb{F}_{p^\ell})$, with $\xi: \alpha \mapsto \alpha^{p^i}$ for some $i$ such that $0\leq i < \ell-1$ (which we now determine). Since conjugation by $x$ inverts $\alpha$, and since $\alpha$ commutes with $x_0$ (because $x_0,\alpha \in \mathbb{F}_{p^\ell}^*$), we have 
    $$
    \alpha^{p^i} = \alpha^\xi = \xi^{-1}\alpha\xi = \xi^{-1}x_0^{-1}\alpha x_0\xi = x^{-1}\alpha x = \alpha^{-1}.
    $$ 
    Hence $\alpha^{p^i+1} = 1$. If $i=0$ then $\alpha^2=1$, so $\alpha = \pm 1$ lies in the prime subfield of $\mathbb{F}_{p^\ell}$ (which is a proper subfield since $\ell>1$), and we have a contradiction. Hence $1\leq i < \ell-1$. Now  $\alpha^{p^i+1} = 1$ implies that $\alpha^{p^{2i}-1}=1$, and we also have $\alpha^{p^\ell-1} = 1$ since   $\alpha \in \mathbb{F}_{p^{\ell}}$. Hence  $\alpha^y=1$ where $y=\gcd(p^\ell-1, p^{2i}-1)$. By Lemma \ref{gcdlemma}, 
    $y=p^j-1$ where $j = \gcd(2i,\ell)$, and it follows that $\alpha$ lies in the subfield $\mathbb{F}_{p^{j}}$ of $\mathbb{F}_{p^\ell}$. Since $\alpha$ lies in no proper subfield, this implies that $j=\ell$, so  $\ell$ divides $2i$. Since $i<\ell-1$, we conclude that $\ell=2i$, and so $\xi$ has order $2$. Thus 
    $$
    \varphi = (x,x^{-1})(12) = (x_0\xi, \xi x_0^{-1})(12) = (x_0, (x_0^{-1})^{p^i})(\xi,\xi)(12).
    $$ 
    Note that $(\xi,\xi)$ simply acts as $\xi$ on all of $V$ and so $(\xi,\xi)$ and $(12)$ commute. Hence 
    $$
    \varphi = (x_0, (x_0^{-1})^{p^i})(12)(\xi,\xi) \in (\GL_1(p^\ell) \wr S_2)\cdot\langle (\xi,\xi)\rangle \leq \GL_2(p^\ell)\cdot2,
    $$ 
    where the last part holds since $\xi$ has order 2.
	
	We now show that $V$ has a proper $G$-invariant subspace, contradicting the irreducibility of $G$. The number of 1-dimensional  $\mathbb{F}_{p^\ell}$-subspaces of $V$ is $p^\ell +1$, and these subspaces are permuted by $G = \langle \sigma, \varphi \rangle$. Furthermore, these subspaces are those of the form $W_\beta = \langle (\beta, 1)\rangle$, for each $\beta \in \mathbb{F}_{p^\ell}$, together with the subspace $\langle (1,0) \rangle$. Since $\sigma$ acts by scalar multiplication it fixes each of these subspaces setwise.  If we then consider the action of $\varphi$ on the subspace $W_{x_0^{-1}} = \langle (x_0^{-1}, 1)\rangle$, where $x_0$ is defined as above, we have
	$$
    (x_0^{-1},1)^\varphi = (x_0^{-1}, 1)^{(x_0, (x_0^{-1})^{p^i})(\xi,\xi)(12)} = (1, (x_0^{-1})^{p^i})^{(\xi,\xi)(12)} =  (1, x_0^{-1})^{(12)} = (x_0^{-1}, 1), 
    $$
	noting that for the third equality we used the fact that $2i = \ell$. Therefore the subspace $W_{x_0^{-1}}$ is fixed by $G$, contradicting the fact that $G$ is irreducible on $V$.  Thus we cannot have $\ell=\phi(r)$, and the proof is complete.
\end{proof}

Finally, we use the above results to describe the basic pairs of unoriented-cycle type which have an abelian socle.

\begin{Theorem}\label{UnorAffineBound}
Let $(\Gamma, G), N, k$, and $r$ be as in Proposition $\ref{unorientedUnique}$ and suppose that $N$ is abelian.	Then one of the following holds
	\begin{enumerate}[(i)]
		\item $k$ is the smallest integer such that $r_0 $ $ |$ $p^k-1$  where $r_0$ is some divisor of $r$; or
		\item $k = 2\ell$ where $\ell$ is the smallest integer such that $r_0 $ $ |$ $p^\ell-1$  where $r_0$ is some divisor of $r$.
	\end{enumerate}
	
	In particular, the largest value for $k$ in either case is $\phi(r)\leq r-1$. 
\end{Theorem}

\begin{proof}
	By Proposition $\ref{unorientedUnique}$, $G$ has a unique minimal normal subgroup $N = \mathbb{Z}_p^k$ and  $G/N \cong D_r$. Then the conjugation action of $G$ on $N$ induces a homomorphism $\rho : G/N \rightarrow \Aut(N) =  \GL_k(p)$ such that $\rho (G/N)$ is isomorphic to either $C_{r_0}$ or $D_{r_0}$ for some divisor $r_0$ of $r$ and ${\rm Ker}(\rho)$ contains $N$ (possibly as a proper subgroup). The minimality of $N$ in $G$ implies that $\rho (G/N)$ is an irreducible subgroup of $\GL_k(p)$.
    If $\rho(G/N) \cong C_{r_0}$ then by Lemma \ref{cyclicrep}, case (i) holds. On the other hand, if  $\rho(G/N) \cong D_{r_0}$ then by Lemma \ref{dihedralrep} exactly one of (i) or (ii) holds. 

    In either case, $r_0$ is coprime to $p$, so $p^{\phi(r_0)} \equiv 1 \pmod{r_0}$. Hence in case (i), $k $ $|$ $ \phi(r_0)$ $ | $ $\phi(r) \leq r-1$, while in  case (ii), $\ell $ $|$. Further, in case (ii), we have
    $\ell \leq \phi(r_0) /2$ by Lemma \ref{dihedralrep},  and so  $k = 2\ell \leq \phi(r_0)\leq \phi(r) \leq r-1$. 
\end{proof}
\section{Proof of Theorem \ref{CycleMainTheorem}}\label{secCycleProof}
We combine the various results from the previous two sections to prove Theorem \ref{CycleMainTheorem}. 

\begin{proof}[Proof of Theorem $\ref{CycleMainTheorem}$]
    Suppose that $(\Gamma, G) \in \OG(4)$ is basic of cycle type and does not have independent cyclic normal quotients. It follows from the discussion in Section \ref{secCyclicNormal} that $(\Gamma, G)$ is either basic of oriented-cycle type, or basic of unoriented-cycle type.
    Suppose first that $(\Gamma, G)$ is basic of oriented-cycle type. Then by Proposition \ref{OrientedTheorem} either subcase 1(a) holds, or the first assertion of subcase 1(b) holds with unique minimal normal subgroup $N$. Moreover, in the latter case, then $k\leq r2^r$ if $N$ is nonabelian, by Theorem \ref{OrientedNonabelianBound}, and $k\leq r$ if $N$ is abelian, by Theorem \ref{OrientedAbelianBound}.
    Now suppose that $(\Gamma, G)$ is basic of unoriented-cycle type. Then by Proposition \ref{unorientedUnique}, the first assertion of Case 2 holds  with unique minimal normal subgroup $N$. Also, either $N$ is nonabelian and $k\leq 2r$, by Theorem \ref{UnorientedNonabelianBound}, or $N$ is abelian and $k \leq r-1$, by  Theorem \ref{UnorAffineBound}. 
    This completes the proof.
  \end{proof}  
    
\section{Constructions of basic pairs }\label{secConst}
In this final section we provide several constructions of infinite families of  basic pairs of cycle type which do not have independent cyclic normal quotients. Our aim is to show that the various bounds on the 
integer $k$ given in Theorem \ref{CycleMainTheorem} are tight by exhibiting  examples of appropriate pairs ($\Gamma, G) \in \OG(4)$ which attain these bounds. 
In the previous sections we analysed the basic pairs of oriented-cycle type and the basic pairs of unoriented-cycle type. 
For each of these two types we saw that the socle could be either abelian or nonabelian and we obtained two  bounds on the number of simple direct factors of the socle according to this distinction. This gives a total of four different types of basic pairs, namely those corresponding to cases 1(b).i, 1(b).ii, 2.i, and 2.ii. of Theorem \ref{CycleMainTheorem}. We will provide constructions of each of these four types. A summary of these constructions along with properties of the values $k$ and $r$ is given in Table \ref{ConstructionTable}. 

\begin{table}[H] 
	\centering
	\begin{tabular}{ l l l l }
		\hline
		$(\Gamma, G)$ Cycle Type & $N = \soc(G)$ &  Construction $\# $ & Properties of $(k,r)$ \\ 
		\hline
		Unoriented & Abelian  & Construction \ref{ConstAffUnoriented} &    $r$ odd prime,  $k = r-1$\\
		Oriented & Abelian & Construction \ref{ConstAffOriented} &  $r \geq 3$,  $k = r$ \\
	
		Unoriented & Nonabelian & Construction \ref{NonabelianUnorientedConst2r} &  $r \geq 3$, $k = 2r$ \\  
		Oriented & Nonabelian  & Construction \ref{NonabelianOrientedConst24} &  $r = 3$, $k = r2^r = 24$  \\
			\hline
	\end{tabular}
\caption{Constructions of basic pairs: $k := $ number of simple direct factors of $N$, $r := |\Gamma_N|$.}\label{ConstructionTable}
\end{table}


Of our four constructions, three provide an infinite family of nonisomorphic basic pairs ($\Gamma, G) \in \OG(4)$ which attain the maximum possible value of $k$ (the number of simple direct factors of $\soc(G)$) for arbitrarily large values of $r$ (where $r$ is the order of the largest cyclic normal quotient). Construction \ref{NonabelianOrientedConst24} however provides an infinite family of basic pairs which attain the maximum possible value of $k$, but only in the case where $r=3$ (and hence $k = r2^r = 24)$. This shows that the bound on $k$  is tight when $r=3$ and we believe this will be true for larger values of $r$ also.

\begin{Problem}\label{prob1}
Decide whether or not the bound $k\leq r2^r$ of Theorem \ref{CycleMainTheorem}, case 1(b).i, is tight when $r\geq 4$.
\end{Problem}

Another interesting question arises here. In \cite{al2017cycle} the question of how many cyclic normal quotients a basic pair can have was considered and it was shown that there is no bound on this number \cite[Theorem 1]{al2017cycle}. Specifically, it was shown that there is no bound on the number of oriented cyclic normal quotients that a basic pair can have. It was therefore asked in  \cite[Problem 1]{al2017cycle}  if the number of unoriented cyclic normal quotients that a basic pair $(\Gamma,G)$ can be arbitrarily large. Our Construction \ref{NonabelianUnorientedConst2r}  shows that there is no bound on this number either, and hence provides an affirmative answer to \cite[Problem 1]{al2017cycle}.

Before giving our constructions, we discuss some theory on coset graphs in $\OG(4)$ and minimal normal subgroups of finite groups which we will use in this section.

\subsection{Coset Graphs in $\OG(4)$}\label{ssCoset}
Given a group $G$, a proper subgroup $S$, and an element $g\in G\backslash S$, the \textit{coset graph} $\Gamma = $ Cos$(G,S,g)$ is the undirected graph having vertex set $V\Gamma = \{Sx : x \in G\}$, and a pair  $\{Sx, Sy\}$ is an edge if and only if $xy^{-1}$ or $yx^{-1} \in SgS$. 
The group $G$, acting by right multiplication on $V\Gamma$, induces a vertex-transitive and edge-transitive group of automorphisms of $\Gamma$. Moreover, every graph which admits a group of automorphisms which is transitive on vertices and edges can be constructed in this way. 

Thus, for every pair $(\Gamma, G) \in \OG(4)$ there exists an $S \leq G$, and a $g\in G$ such that $\Gamma \cong $ Cos$(G, S, g)$. Specifically, for a vertex $\alpha \in V\Gamma$, take $S:= G_\alpha$ and take $g$ to be an element of $G$ mapping $\alpha$ to one of its neighbours; note that $\alpha^{g^2} \neq \alpha$ as otherwise $g$ would reverse the arc $(\alpha,\alpha^g)$ and $G$ would be arc-transitive. 
 Lemma \ref{CosetOriented} below gives specific conditions as to when a graph-group pair $(\Gamma, G)$ constructed in this way is a member of $\OG(4)$. This result was outlined in \cite[Section 5]{al2015finite}, and a good in-depth breakdown is given in \cite[Section 2]{li2004primitive}. In summary, conditions (i)-(iv) of Lemma \ref{CosetOriented} ensure that: (i) $G$ is faithful on $V\Gamma$; (ii) $G$ preserves an orientation of $E\Gamma$; (iii) $\Gamma$ is 4-valent; and (iv) $\Gamma$ is connected.  

\begin{Lemma}\label{CosetOriented}
	Suppose that $\Gamma = $ \emph{Cos}$(G, S, g)$ for some group $G$, some proper subgroup $S < G$ and some element $g \in G$. Then $(\Gamma, G) \in \OG(4)$ if the following four properties hold:
	\begin{enumerate}[]
	\item 	\emph{ (i)} $S$ is core-free in $G$, \hspace{.2cm} 	\emph{ (ii)} $g^{-1} \notin SgS$, \hspace{.2cm} 	\emph{(iii) }$|S: S\cap S^g| = 2$,   \hspace{.2cm} 	\emph{(iv)} $\langle S, g\rangle = G$.
	\end{enumerate}
\end{Lemma}

All graphs constructed in this section are described as coset graphs of certain groups. We will thus heavily rely on Lemma \ref{CosetOriented}.
Since we will also need to show that pairs constructed in this way are basic of cycle type, we will often use the following result. 

\begin{Lemma}\label{CosetCycleType}
	Suppose that $(\Gamma, G) \in \OG(4)$, with $\Gamma = $ \emph{Cos}$(G, S, g)$ where $S < G$ and $g \in G$. Suppose further that $G$ has a unique minimal normal subgroup $N$. Then
	
	\begin{enumerate}[(a)]
	\item $(\Gamma, G)$ is basic of oriented-cycle type if there are two distinct $SN$-cosets in $G$, say  $SNx$ and $SNy$ with $x,y \in G\backslash SN$, such that, $SgS \subset SNx$, and $Sg^{-1}S \subset SNy$.
	\item $(\Gamma, G)$ is basic of unoriented-cycle type if there are two distinct $SN$-cosets in $G$, say  $SNx$ and $SNy$ with $x,y \in G\backslash SN$, such that 
	each of $SNx$ and $SNy$ contains half of the elements of $SgS$ and half of the elements of $Sg^{-1}S$.
		\end{enumerate}
\end{Lemma}
\begin{proof}
	The $N$-orbit in $V\Gamma$ containing the vertex $S \in V\Gamma$ is the set of all right $S$-cosets contained in the subgroup $SN$ of $G$. More generally (since $N\unlhd G$), for each $x\in G$, the $N$-orbit  containing the vertex $Sx$ is the set of right $S$-cosets contained in the $SN$-coset $SNx = SxN$. Thus the set of $N$-orbits in $V\Gamma$ is in bijection with the set of right $SN$-cosets in $G$ (which forms a partition of $G$ invariant under the right multiplication action of $G$). 
 
 	 By construction, the vertex $S$ has two out-neighbours which are right $S$-cosets contained in $SgS$, and two in-neighbours which are right $S$-cosets contained in $Sg^{-1}S$. If either of the conditions $(a)$ or $(b)$ holds then these neighbours are all contained in exactly two $SN$-cosets, $SNx$ and $SNy$, for some $x, y \in G\backslash SN$. Moreover each vertex $Sn \in SN$ will also only have neighbours in $SNx$ and $SNy$. Hence $\Gamma_N$ will be 2-valent and so $\Gamma_N$ will be a cycle. So since by assumption $N$ is the unique minimal normal subgroup of $G$, and since $\Gamma_N$ is a cycle, it follows that $(\Gamma, G)$ is basic of cycle type.
	
	Note that multiplication by an element $s \in S$ fixes the vertex $S$ and fixes or swaps its two out-neighbours (both contained in $SgS$), and fixes or swaps its two in-neighbours (both contained in $Sg^{-1}S$). 
	Therefore if $(a)$ holds then  multiplication by an element of $S$ will fix each of $SN$, $SNx$ and $SNy$ setwise and hence fix all vertices of $\Gamma_N$. In this case $\Gamma_N$ is a $G$-oriented cycle. 
	On the other hand, if $(b)$ holds then since $(\Gamma, G)\in \OG(4)$ there is an element $s\in S$ which interchanges the two out-neighbours of $S$. These out-neighbours are both contained in $SgS$ with one contained in $SNx$ and the other in $SNy$.
	Multiplication by this element $s \in S$ therefore fixes the $N$-orbit $SN$, and swaps its two adjacent $N$-orbits in $\Gamma_N$, $SNx$ and $SNy$. Thus the quotient graph $\Gamma_N$ is a $G$-unoriented cycle.
\end{proof}
    
 \subsection{Minimal Normal  Subgroups of Finite Groups}\label{secMinNormal}
Recall that in this paper all groups are assumed to be finite, and hence each group will have at least one minimal normal subgroup.
%
Each minimal normal subgroup $N$ of a finite group $G$ is characteristically simple, and hence is isomorphic to a direct product of isomorphic simple groups \cite[Lemma 3.14]{praeger2018permutation}.  Any two distinct minimal normal subgroups must intersect trivially and centralise each other \cite[Lemma 3.4]{praeger2018permutation}. Moreover, for a nonabelian characteristically simple normal subgroup $N$ of a group $G$, where $N = T_1 \times \cdots \times T_k$ and the $T_i$ are nonabelian simple groups, the group $G$ acts by conjugation on the set $\{T_1,\dots, T_k\}$, and is transitive on this set if and only if $N$ is a minimal normal subgroup of $G$ \cite[Corollary 4.16]{praeger2018permutation}.

The \textit{socle} of a group $G$, denoted $\soc(G)$ is the subgroup generated by all of the minimal normal subgroups of $G$. Since all minimal normal subgroups of $G$ intersect trivially and commute, it follows that $\soc(G)$ can be written as a direct product $\prod_i N_i$ where the $N_i$ are minimal normal subgroups of $G$. Thus $\soc(G)$ can be written as a direct product of finite simple groups. 


Let $G = G_1 \times \cdots \times G_k$ be a group with direct product decomposition $S = \{G_1, \dots, G_k\}$, and for each $k$ such that $1\leq i\leq k$, let $\pi_i$ denote the projection map $\pi_i: G \rightarrow G_i$ given by $(g_1, \dots, g_k) \mapsto g_i$. We say that a nontrivial subgroup $H$ is a \textit{strip} of $G$ if for each $i$, either $\pi_i(H)  = 1$ or $\pi_i(H)\cong H$. The \textit{support} of a strip $H$ is defined as Supp$(H) := \{G_i: \pi_i(H) \neq 1\}$. We say that a strip $H$ is \textit{nontrivial} if $|\Supp (H)| >1$, and  is \textit{full} if $\pi_i(H) = G_i$ for all $G_i \in \Supp(H)$; and two strips are \emph{disjoint} if their supports are disjoint.  Details regarding the above concepts can be found in \cite[Section 4.4]{praeger2018permutation}. Further, a subgroup $H \leq G$ is $S$-\textit{subdirect} if $\pi_i(H) = G_i$ for all $i \in \{1,\dots, k\}$. If the decomposition $S$ is clear from the context, we simply say that $H$ is a subdirect subgroup of $G$.
 The following useful result describes subdirect subgroups of finite nonabelian characteristically simple groups.

 

\begin{Proposition}\emph{\textbf{(Scott's Lemma) }}\emph{\cite[Theorem 4.15 (iii)]{praeger2018permutation}}\label{strips}
Let $R = \{1,\dots,k\}$ where $k \geq 1$, and let $N= \prod_{i \in R} T_i$ where the $T_i$ are isomorphic finite nonabelian simple groups.	
 If $H$ is a subdirect subgroup of $N$ then $H$ is a direct product of pairwise disjoint full strips of $N$.
\end{Proposition}
Using Proposition \ref{strips}, we obtain the following method for showing that a subgroup of a nonabelian characteristically simple group $N$ is equal to $N$. 

\begin{Lemma}\label{stripmethod}
	Let $R, k, N$ be as in Proposition~$\ref{strips}$, and let $H \leq N$. Then $H = N$ if the following  two conditions hold:
	\begin{enumerate}[(a)]
		\item $\pi_i(H) = T_i$ for all $i \in R$; and
		\item no direct factor of $H$ is a nontrivial full strip of $N$.
	\end{enumerate}
\end{Lemma}
\begin{proof}
	Condition (a) together with Proposition \ref{strips} implies that $H$ is a direct product of pairwise disjoint full strips of $N$. Condition (b) then implies that each of these strips is trivial and hence that $H = T_1\times \cdots \times T_{k} = N$.
\end{proof}

Finally, we give the following lemma which follows from \cite[Lemma 4.10]{praeger2018permutation}. This is especially useful for verifying condition (b) of Lemma \ref{stripmethod}. 
\begin{Lemma}\label{stripAuto}
    Let $R, k, N$ be as in Proposition~$\ref{strips}$, and let $T$ be a finite nonabelian simple group such that each $T_i \cong T$. Let $F$ be a nontrivial full strip of $N$ and let $J:= \{j \in R : \pi_j(F) \neq 1\}$. Then for each $j \in J$ there is an automorphism $\psi_j \in \Aut(T)$ such that 
	
	$$F = \{(h^{\phi_1},\dots, h^{\phi_k}) : h \in T\}, \hbox{ where } h^{\phi_i} = \begin{cases} & h^{\psi_i}, \hbox{ if } i \in J \\ &1, \hspace{3.4mm}\hbox{ if } i \notin J.\end{cases}$$
	
	In particular, for any $x \in F$ and any $i,j\in J$, there is an automorphism $\psi \in \Aut(T)$ such that  $\pi_i(x)^\psi = \pi_j(x)$. 
\end{Lemma}
For  $N= \prod_{i \in R} T_i$ as in Proposition~$\ref{strips}$, and $F$ a nontrivial full strip of $N$ with  $J:= \{j \in R : \pi_j(F) \neq 1\}$, we  say that $F$ is a \textit{diagonal subgroup} of the subproduct $\prod_{i \in J} T_i$.   

\smallskip
We are now ready to give our constructions. Due to some similarities between them, we will first provide our  constructions of basic pairs which have an abelian socle in Section \ref{secCycleConstructionAbelian}, before moving on to the nonabelian case in Section \ref{secCycleConstructionNonbelian}.

\subsection{Constructions with Abelian Socle}\label{secCycleConstructionAbelian}
Our first construction provides an infinite family of graph-group pairs $(\Gamma, G) \in \OG(4)$ of unoriented cycle-type with an abelian unique minimal normal subgroup. Pairs of this kind correspond to Case 2.ii. of Theorem \ref{CycleMainTheorem}.

\begin{Construction}\label{ConstAffUnoriented}
	Let $p, r$ be distinct odd primes such that $p <r$ and $p$ is a primitive root modulo $r$. 
	Let $N = \mathbb{Z}_p^{r-1}$ and  let $f$ be the polynomial $f(X) = X^{r-1} + X^{r-2} + \dots X +1 \in \mathbb{Z}_p[X]$ of degree $r-1$.
	Define two $(r-1)\times (r-1)$ matrices $\sigma, \varphi \in \GL_{r-1}(p)$ as follows; $\sigma$ is the companion matrix of $f$, and $\varphi$ is the matrix whose only non-zero entries are $-1$ on the anti-diagonal, that is 
	
	$$ \sigma = \begin{bmatrix}
	0      & 0      & \dots      & 0 & -1 \\
	1     & 0      & \cdots    & 0 & -1 \\
	0      & 1     & \cdots   &  0& -1 \\
		\vdots & \vdots & \smash{\ddots} & \vdots  & \vdots \\
	0  & 0    & \cdots & 1  & -1
	\end{bmatrix} , 
	\hbox{ and }
	\varphi = \begin{bmatrix} 
	0 & &  &  & -1  \\
	&  &  &  -1 & \\
	& &   \reflectbox{$\ddots$} & & \\
	& -1  &  & & \\
	-1 &    &    &   & 0
	\end{bmatrix}.$$
	$$ $$
	
	Let $G := N\rtimes \langle \sigma, \varphi\rangle$ and note that $\langle \sigma , \varphi\rangle \cong D_r$ (and $|\varphi|=2$). Let $n:= (1, 0 ,\dots, 0) \in N$, and let $g := n \sigma$ and $s := \varphi$. Finally define $\Gamma:= \Cos(G, S, g)$ where $S = \langle s \rangle\cong \mathbb{Z}_2$.
\end{Construction}

\begin{Remark} \label{ArtinRemark}
	Artin's conjecture, which was communicated to Helmut Hasse by Emil Artin in 1927 \cite[p.4]{moree2012artin},   states that if $k\in \mathbb{Z}$ is not -1 or a perfect square, then $k$ is a primitive root modulo infinitely many primes. By \cite[Corollary 2]{heath1986artin}, there are at most two (positive) primes for which Artin's conjecture fails. In particular, at least one element of $\{3,5,7\}$ is a primitive root modulo infinitely many primes.  Hence there are infinitely many prime pairs $(p,r)$ fitting the criteria of Construction \ref{ConstAffUnoriented}.
\end{Remark}

\begin{Lemma}
	Let $(\Gamma, G)$ be as in Construction $\ref{ConstAffUnoriented}$. Then $(\Gamma, G) \in \OG(4)$ and is basic of unoriented-cycle type. Moreover, $N$ is the unique minimal normal subgroup of $G$, $N\cong \mathbb{Z}_p^{r-1}$, and $\Gamma_N \cong \mathbf{C}_r$. 
 \end{Lemma}

\begin{proof}
	We will use Lemmas \ref{CosetOriented} and \ref{CosetCycleType} to prove this claim. We begin by showing that $(\Gamma, G) \in \OG(4)$. 
	
	In what follows we will use multiplicative notation for multiplication in the group $G$. To make calculations more intuitive however, we will use additive notation when composing two elements of the subgroup $N = \mathbb{Z}_p^{r-1}$. Note first that 
    $$
    s^g = \sigma^{-1}(-n)\varphi n\sigma = \sigma^{-1}(-n + n^\varphi) \sigma^\varphi \varphi = \sigma^{-1}(-n + n^\varphi) \sigma^{-1} \varphi= (-n + n^\varphi)^\sigma \sigma^{-2} \varphi,
    $$ 
    and if this were equal to $s = \varphi$ then we would have $(-n + n^\varphi)^\sigma = \sigma^{2}\in N\cap \langle \sigma\rangle = 1$, whereas $|\sigma|\geq 3$. Thus $s^g\ne s$. Since $S \cong \mathbb{Z}_2$, this implies that $S \cap S^g = 1$, and hence $S$ is core-free in $G$ and $|S: S\cap S^g| = 2$. Next note that $SgS = \{g, g\varphi ,  \varphi g , g^\varphi\}$, and that  $g^2 = n\sigma n\sigma = (n+ n^{\sigma^{-1}}) \sigma^2$. Then $g^2 \neq 1$, so $g^{-1}\neq  g$; and $g^2 \neq \varphi$, so $g^{-1} \neq g\varphi$ and $g^{-1} \neq \varphi g$. Thus $g^{-1} \in SgS$ if and only if $g^{-1} = g^\varphi$. However, $g^\varphi = \varphi n\sigma \varphi = n^\varphi \sigma^{-1}$, and so $g^\varphi g = n^\varphi\sigma^{-1}n\sigma = n^\varphi + n^\sigma$. Now $n^\varphi + n^\sigma = (0,\dots,0,-1) + (0, 0,\dots, -1) \neq 0_N$  since $p \neq 2$, and we conclude that $g^\varphi \neq g^{-1}$. Hence $g^{-1} \notin SgS$. Thus properties (i) - (iii) of Lemma \ref{CosetOriented} hold. 
	
	In order to establish property (iv), that  $\langle g, s \rangle = G$, we first compute the powers of $g$. We have already seen that $g^2 =  (n+ n^{\sigma^{-1}}) \sigma^2$. Inductively, for $i \geq 2$, if $g^i = (n+n^{\sigma^{-1}}+\cdots+ n^{\sigma^{-(i-1)}})\sigma^{i}$, then 
	\begin{align*}
	g^{i+1} &= g^i(n\sigma)\\ &= (n+n^{\sigma^{-1}}+\cdots+ n^{\sigma^{-(i-1)}})\sigma^{i}n\sigma \\&= (n+n^{\sigma^{-1}}+\cdots+ n^{\sigma^{-(i-1)}}+n^{\sigma^{-i}})\sigma^{i+1}.
	\end{align*}
	Hence for each $i \in [1,r]$ we may write $g^i = n_i\sigma^{i}$, where $n_i:= n+n^{\sigma^{-1}}+\cdots+ n^{\sigma^{-(i-1)}}.$ Next for each $i \in [1,r-1]$ we will let $X_i$ denote the element $sg^isg^i \in \langle s,g\rangle$. More explicitly, 
	$$X_i := sg^isg^i = \varphi n_i\sigma^i \varphi n_i \sigma^i = \varphi n_i \varphi \sigma^{-i}n_i\sigma^{i} = n_i^\varphi + n_i^{\sigma^{i}}.$$
	We claim that the set $\{X_i: 1 \leq i \leq r-1\}$ generates $N$. We first the compute $n^{\sigma^i}$. Since $n = (1, 0,\dots, 0)$, $n^\sigma$ is the first row of the matrix $\sigma$, namely $(0,0,\dots, -1)$. This implies that $n^{\sigma^{2}}$ is $(0,\dots, -1,1)$, namely the last row of $\sigma$ multiplied by $-1.$ Now $n^{\sigma^{3}}$ is equal to the $(r-2)$nd row of $\sigma$ subtracted from the $(r-1)$st row, that is $n^{\sigma^3} = (0, \dots, -1,1,0)$. Continuing in this way we  see that, for $3 \leq i \leq r-1$, $n^{\sigma^i}$ is the vector with $-1$ in the $(r-i)$th position, $+1$ in the $(r-i+1)$th position, and $0$ in all other entries, that is 
    $$
    n^{\sigma^{i}}= (0,\dots, 0, \underbrace{-1}_{\text{\normalfont $(r-i)$th entry}},\underbrace{1}_{\text{\normalfont $(r-i+1)$th entry}}, 0, \dots, 0), \hbox{ \hspace{5mm}for } 2 \leq i \leq r-1.
    $$ 
    From this it easily follows that 
	\begin{equation}\label{ni}
	n^{\sigma^{-i}}= (0,\dots, 0, \underbrace{-1}_{\text{\normalfont $i$th entry}},\underbrace{1}_{\text{\normalfont $(i+1)$th entry}}, 0, \dots, 0), \hbox{ \hspace{5mm}for } 1 \leq i \leq r-2.
	\end{equation}
	
	Let us  consider the element $n_i \in N$. Now $n_i = n+n^{\sigma^{-1}}+\cdots+ n^{\sigma^{-(i-1)}}$ and so by (\ref{ni}) it is easy to see that for $1\leq i \leq r-2$, $n_i$ is a vector with a 1 in the $i$th entry and all other entries 0. So $$n_i = (0,\dots, 0, \underbrace{1}_{\text{\normalfont $i$th entry}},0, \dots, 0), \hbox{ \hspace{5mm}for } 1 \leq i \leq r-1.$$ It follows that $n_i^\varphi = -n_{r-i}$.
	
	On the other hand $n_i^{\sigma^i} =  (n+n^{\sigma^{-1}}+\cdots+ n^{\sigma^{-(i-1)}})^{\sigma^i } = n^{\sigma^{i}} + n^{\sigma^{i-1}}\cdots + n^\sigma$, so using the equations above we  check that $n_i^{\sigma^i}$ is the vector with $-1$ in the $(r-i)$th entry and all other entries 0. Hence  $n_i^{\sigma^i} = n_i^\varphi = -n_{r-i}$, for each $i \in [1,r-1]$.
	Thus $X_i = -2\cdot n_{r-i}$, and since the vectors $n_i$ are clearly linearly independent, so are the vectors $X_i$ for $i \in [1,r-1]$. Hence $N \leq \langle g, s\rangle$, and it follows that $\langle g, s \rangle = G$. Therefore $(\Gamma, G) \in \OG(4)$ by Lemma \ref{CosetOriented}.
	
	Next we claim that $\langle \sigma \rangle$  is irreducible on $N$. (From this it follows that $G$ is irreducible on $N$.) Since $\sigma$ is the companion matrix of $f$, by Lemma \ref{irredcorr} it suffices to show that $f$ is an irreducible polynomial over $\mathbb{Z}_p[X]$.
	To see that this is the case, note that $X^{r}-1 = (X-1)f(X)$. Thus the roots of $f(X)$ are the $r$th roots of unity in $\mathbb{Z}_p$. So let $\zeta$ be an $r$th root of unity and take an irreducible divisor $q(X)$ of $f(X)$ which has $\zeta$ as a root. Then all elements of $R:= \{\zeta, \zeta^p, \zeta^{p^2}, \dots \}$ are roots of $q(X)$. Of course $R$ must be finite since the degree of $q(X)$ is at most $r-1$.	So take the smallest integer $\ell \geq 1 $ such that $\zeta^{p^{\ell}} = \zeta$, and hence such that $\zeta^{p^{\ell-1}} = 1$. Since $|\zeta| = r$ it follows that $r \mid  p^\ell -1$, and so $\ell$ is the smallest integer such that $p^\ell \equiv 1$ mod $r$, that is to say, $\ell$ is the order of $p$  in $\mathbb{Z}_{r}^*$, which is $r-1$ (since $p$ is a primitive root modulo $r$). Thus $q(x)$ has degree $r-1$, and so $f(X) = q(X)$ is irreducible in $\mathbb{Z}_p[X]$. 	This proves the claim that  $\langle \sigma \rangle$, and hence also $G$, acts irreducibly on $N$. 
 
    It follows that $N$ is a minimal normal subgroup of $G$. Our computations above showed that $N\cong \mathbb{Z}_p^{r-1}$ and that $D_r=\langle \sigma, \varphi\rangle$ acts faithfully by conjugation on $N$. Thus $N$ is self-centralising in $G$, and hence is the unique minimal normal subgroup of $G$. 	
	Finally, computing 
    $$
    SgS = \{g, sg, gs, sgs\} = \{n\sigma, \varphi n\sigma, n\sigma\varphi, \varphi n \sigma \varphi\} = \allowdisplaybreaks \{n\sigma, \varphi n\sigma, \varphi n^\varphi\sigma^{-1}, n^\varphi\sigma^{-1}\},
    $$ 
    we see that half of these four elements are contained in $SN\sigma$ and the other half are contained in $SN\sigma^{-1}$. Moreover 
	\begin{align*}
	   Sg^{-1}S &= \{g^{-1}, sg^{-1}, g^{-1}s, sg^{-1}s\} =  \{\sigma^{-1}n^{-1}, \varphi \sigma^{-1}n^{-1}, \sigma^{-1}n^{-1}\varphi, \varphi \sigma^{-1}n^{-1}\varphi\} \\&= \{(n^{-1})^\sigma \sigma^{-1}, \varphi (n^{-1})^\sigma \sigma^{-1}, \varphi(n^{-1})^{\sigma\varphi} \sigma, (n^{-1})^{\varphi\sigma^{-1}}\sigma\},
	\end{align*}
	 and the same is true for $Sg^{-1}S$. Hence, by Lemma \ref{CosetCycleType} (b), $(\Gamma, G)$ is basic of unoriented-cycle type, and $\Gamma_N$ is the unoriented cycle $\mathbf{C}_r$. 
\end{proof}
Our next construction provides pairs of the kind described in Case 1(b)ii. of Theorem \ref{CycleMainTheorem}.

\begin{Construction}\label{ConstAffOriented}
	Let $r >2$ and let $p$ be an odd prime which does not divide $r-2$. Let $N = \mathbb{Z}_p^r$ and
	 define the matrices $\sigma, \varphi_1, \dots, \varphi_r \in \GL_r(p)$ as follows; $\sigma$ is the permutation matrix of order $r$ which cyclically permutes the $r$ standard basis vectors, and for $i \in [1,r]$, $\varphi_i$ is the diagonal matrix with $-1$ in the ($i,i)$ entry, and $1$ in all remaining diagonal entries. That is 

$$ \sigma = \begin{bmatrix}
0 &1     & 0      & \cdots    & 0   \\
0 &0      & 1     & \cdots   &  0 \\
0 &\vdots & \vdots & \smash{\ddots} & \vdots  \\
0 &0  & 0    & \cdots & 1  & \\
1 & 0  & 0&  \dots    & 0 \\
\end{bmatrix} , 
\hbox{ and }
\varphi_i = \begin{bmatrix} 
  1&0&\cdots&\cdots &0\\
\vdots&\ddots&&\cdots &0\\
0&&-1_{(i,i)}&\cdots &0\\
\vdots&&&\ddots&\vdots\\
0&0&0&\cdots &1
\end{bmatrix}, \hbox{ for $i \in [1,r].$ } $$

Notice that conjugation by $\sigma$ cyclically permutes the elements in the set $\{\varphi_1,\dots, \varphi_r\}$ according to the rule $\varphi_i^\sigma = \varphi_{i+1}$ where $\varphi_r^\sigma = \varphi_1$.

Let $S \leq \GL_r(p)$ be the group generated by $\{\varphi_1,\dots, \varphi_r\}$ and note that  \newline $S = \langle \varphi_1,\dots, \varphi_r \rangle \cong \mathbb{Z}_2^r$. Define $G := (N\rtimes S) \rtimes \langle \sigma \rangle$ with $\langle S, \sigma\rangle$ acting on $N$ by right multiplication. Note that $G/N \cong \langle S, \sigma \rangle \cong \mathbb{Z}_2^r \rtimes \mathbb{Z}_r$. Finally let $n:= (0, 0 ,\dots,0, 1) \in N$, let $g := n \sigma\in G$, and define $\Gamma:= \Cos(G, S, g)$.
\end{Construction}

\begin{Lemma}
	Let $(\Gamma, G)$ be as in Construction $\ref{ConstAffOriented}$. Then $(\Gamma, G) \in \OG(4)$, and is basic of oriented-cycle type. Moreover, $G$ has a unique minimal normal subgroup $N$ isomorphic to $\mathbb{Z}_p^{r}$ and  $\Gamma_N \cong \mathbf{C}_r$. 
 \end{Lemma}

\begin{proof}
	First we use Lemma \ref{CosetOriented} to prove that $(\Gamma, G) \in \OG(4)$. For $1 \leq i \leq r$,  let $e_i$ denote the $ith$ standard basis vector of $\mathbb{F}_p^r$.  We identify the elements of the subgroup $N$ of $G$ with the vectors from $\mathbb{F}_p^r$ in the natural way, and we use additive notation for composing elements within the subgroup $N$. 
	To check that $S$ is core-free in $G$, consider $h = (1,1,\dots, 1) \in N$. For $s \in S$ we have $s^h = (h^{-1})sh = h^{-1}h^s \cdot s = (-h + h^s) \cdot s$, and this element is contained in $S$ if and only if $h^s = h$ which, in turn, holds only if $s = 1$. Therefore $S^h \cap S = \{1\}$ and so $S$ is core-free in $G$.
	
	To see that $g^{-1} \notin SgS$ we note that $g^{-1} = \sigma^{-1}n^{-1} = (n^{-1})^\sigma \sigma^{-1}\in NS \sigma^{-1}$, while every element $x \in SgS$ is of the form $x = s_1 n\sigma s_2 = s_1ns_2^{\sigma^{-1}}\sigma\in NS\sigma$ for some $s_1, s_2 \in S$. Since $r\geq3$ the two cosets $NS \sigma^{-1}$ and $NS \sigma$ are distinct, and hence $g^{-1} \notin SgS$. 
	
	Next we show that $|S: S\cap S^g| = 2$.  Note that $n = e_r$ and $e_i^\sigma = e_{i+1}$ with $e_r^\sigma = e_1$, and further 
    $$
    e_j^{\varphi_i} = \begin{cases}
	-e_j, \hbox{ if $i = j$}\\
	e_j, \hspace{3mm}\hbox{ if $i\neq j$}
	\end{cases} .
    $$ 
    We determine $S^g$ by computing $\varphi_i^g$ for each $i \in [1,r]$.  Here, interpreting $\varphi_{r+1}$ as $\varphi_1$,  we have 
	$$
    \varphi_i^g = \sigma^{-1}n^{-1}\varphi_i n \sigma = (n^{-1})^\sigma \varphi_i^\sigma n^\sigma = (-e_1) \varphi_{i+1} e_1 = (-e_1 + e_1^{\varphi_{i+1}}) \varphi_{i+1}.
        $$
	For $1\leq i <r$, we have $e_1^{\varphi_{i+1}} = e_1$ and hence $\varphi_i^g = \varphi_{i+1}$, while  $\varphi_r^g = (-e_1 + e_1^{\varphi_1})\varphi_1 = -2e_1 \varphi_1 \notin S$. It follows that $S \cap S^g=\langle \varphi_2, \dots \varphi_r \rangle$ and $|S: S\cap S^g| = 2$.
	
	Finally to show that $\langle S, g \rangle = G$ it is sufficient to show that $N \leq \langle S, g\rangle$ from which the result easily follows since $\langle S, g \rangle = \langle S, \sigma \rangle$ modulo $N$. First we show that $\langle S, g \rangle$ contains a nontrivial element of $N$, namely $g^r$, which  we  compute  as follows (noting that $\sigma^r = 1$): 
    \[
g^r =\underbrace{(n\sigma) (n\sigma)\dots (n\sigma)}_{\text{\normalfont $r$ times}} = n n^{\sigma^{-1}}n^{\sigma^{-2}}\cdots n^{\sigma^{-(r-1)}}\sigma^r = e_r + e_{r-1} + e_{r-2} + \cdots+ e_{1} = (1,1,\dots,1)\in N.
    \]
	Since $N\unlhd G$,  $\varphi_ig^r \varphi_i\in N$ for each $i \in [1,r]$. We claim that the set $X = \{\varphi_ig^r \varphi_i: i \in [1,r]\}$ generates $N = \mathbb{Z}_p^r$. Note that, for $i\in [1,2]$, 
	$$\varphi_ig^r\varphi_i =  (1,1,\dots, 1)^{\varphi_i} = (1,1,\dots,1, \underbrace{-1}_{\text{\normalfont $i$th entry}},1,\dots, 1),
        $$ 
	so the $\mathbb{F}_p$-span of $X$ is the row space of the following $r\times r$ matrices over $\mathbb{F}_p$: 
    $$
    \begin{bmatrix}
-1 &1     & 1      & \cdots  &1  & 1   \\
1 &-1      & 1     & \cdots  &1 &  1 \\
\vdots &\vdots & \vdots & \smash{\ddots}&\vdots & \vdots  \\
1 &1  & 1    & \cdots & -1  &1 \\
1 & 1  & 1&  \dots    &1& -1 \\
\end{bmatrix} \ \equiv \  
    \begin{bmatrix}
-1 &1     & 1      & \cdots  &1  & 1   \\
2 &-2     & 0     & \cdots   &0 &  0 \\
\vdots &\vdots & \vdots & \smash{\ddots}&\vdots & \vdots  \\
2 &0  & 0    & \cdots & -2  &0 \\
2 & 0  & 0&  \dots    &0& -2 \\
\end{bmatrix} \ \equiv \  
    \begin{bmatrix}
r-2 &0     & 0      & \cdots  &0  & 0   \\
2 &-2     & 0     & \cdots   &0 &  0 \\
\vdots &\vdots & \vdots & \smash{\ddots}&\vdots & \vdots  \\
2 &0  & 0    & \cdots & -2  &0 \\
2 & 0  & 0&  \dots    &0& -2 \\
\end{bmatrix}    
$$
where $\equiv$ denotes row equivalence. All three of these matrices have the same row space, and it is clear that the third matrix is nonsingular since $p$ does not divide $r-2$. It follows that the set $X$ spans $N= \mathbb{Z}_p^r$, and hence $N\leq \langle S, g \rangle = G$, so $\langle S, g \rangle = G$ and   $(\Gamma, G)\in \OG(4)$ by Lemma \ref{CosetOriented}.
	
	
	Our next goal is to show that $N$ is the unique minimal normal subgroup of $G$ so that we can apply Lemma \ref{CosetCycleType}. We again treat $N$ as a vector space $V = \mathbb{F}_p^r$, and show that $G/N \cong \langle S, \sigma \rangle$ acts irreducibly on $V$. Let $G^+ = \langle S, \sigma \rangle$. It is sufficient to show that the subspace $\langle v^{G^+} \rangle$ spanned by the $G^+$-orbit of any nonzero vector $v \in V$ is equal to $V$. Such a vector has the form $v = \alpha_1 e_1 + \cdots + \alpha_r e_r$, where the $\alpha_i \in \mathbb{F}_p$ and at least one $\alpha_i\ne 0$. Without loss of generality we may assume that $\alpha_1 \neq 0$. (For if $\alpha_1=0$ and $\alpha_j \neq 0$ for some other $j$, then repeated applications of $\sigma \in G^+$ to $v$ produces another vector, say $w = \beta_1e_1 +\cdots \beta_re_r$, which lies in $v^{G^+}$ and which has $\beta_1 = \alpha_j \neq 0$, and we may proceed with this $w$.) We will show that each $e_i$ is contained in $\langle v^{G^+} \rangle$.  Notice that 
    $$
    v^{\varphi_2} = \alpha_1e_1 -\alpha_2e_2 +\cdots + \alpha_r e_r,\ \hbox{so that}\ v+v^{\varphi_2} = 2\alpha_1 e_1 + 2\alpha_3 e_3 +\cdots +2\alpha_r e_r\in \langle v^{G^+} \rangle.
    $$
	Continuing in this way, we recursively define $v_1:= v$ and $v_i := v_{i-1}+ v_{i-1}^{\varphi_i}$ for $i \in [2,r]$; note that all these vectors lie in $\langle v^{G^+} \rangle$. It is easily checked that $v_r = 2^{r-1}\alpha_1e_1$, which is non-zero since $p$ is  odd.  Hence $e_1 \in \langle v^{G^+} \rangle$; and also  $e_1^{\sigma^{i-1}} = e_i \in \langle v^{G^+} \rangle$ for each $i \in [1,r]$. Thus $\{e_1, e_2, \dots, e_r\}\subset \langle v^{G^+} \rangle$, and hence $\langle v^{G^+} \rangle=V$. Therefore $G^+$ acts irreducibly on $V$, and so $N$ is a minimal normal subgroup of $G$. 
    It follows from Construction~\ref{ConstAffOriented} that  $G^+ = \langle S, \sigma \rangle$ acts faithfully by conjugation on $N$, and hence $N$ is self-centralising in $G$. This implies that $N$ is the unique minimal normal subgroup of $G$.
    	
	Finally we show that (i) $SgS \subset SN\sigma$ and (ii) $Sg^{-1}S\subset SN\sigma^{-1}$, which together imply that $(\Gamma, G) \in \OG(4)$ is basic of oriented-cycle type, by Lemma \ref{CosetCycleType}. First, each element $a\in SgS$ has the form $a = s_1n\sigma s_2$ for some $s_1, s_2 \in S$, so 
    $$
    a = s_1n\sigma s_2 = s_1ns_2^{\sigma^{-1}}\sigma = s_1s_2^{\sigma^{-1}}n^{s_2^{\sigma^{-1}}}\sigma \in SN\sigma
    $$
    and condition (i) holds.	Similarly each $b \in Sg^{-1}S$ has the form $b = s_1\sigma^{-1}n^{-1} s_2$ with $s_1, s_2 \in S$, so 
    $$
    b =s_1\sigma^{-1}n^{-1} s_2 = s_1(n^{-1})^\sigma \sigma^{-1} s_2 = s_1(n^{-1})^\sigma  s_2^\sigma \sigma^{-1} =  s_1 s_2^\sigma (n^{-1})^{\sigma s_2^\sigma}  \sigma^{-1}  \in SN\sigma^{-1}.
    $$
	By Lemma \ref{CosetCycleType} $(\Gamma, G) \in \OG(4)$ is basic of oriented-cycle type, and $\Gamma_N$ is the oriented cycle $\mathbf{C}_r$. 
\end{proof}

\subsection{Constructions with Nonabelian Socle}\label{secCycleConstructionNonbelian}

Our final two constructions both give basic pairs $(\Gamma, G)\in \OG(4)$ where $\soc(G)$ is nonabelian. In both cases $G$ has a unique nonabelian minimal normal subgroup $N \cong T^k$ where $T$ is a nonabelian simple group. Here we will make heavy use of the theory discussed in Section \ref{secMinNormal}.

\subsubsection{$(\Gamma, G)$ is basic of unoriented-cycle type}

Our third construction produces basic pairs $(\Gamma, G)$ which are basic of unoriented-cycle type.

\begin{Construction}\label{NonabelianUnorientedConst2r}
	Take any integer $r \geq 3$ and let $T$ be a nonabelian simple group generated by two elements $a$ and $b$, such that $|a|=2, |b|\geq 3$, and the orders $|ab|$ and $|ba|$ are different from both $|a|$ and $|b|$.
	Consider the group $T \wr S_{2r}$ with $S_{2r}$ acting by permuting the ${2r}$ simple direct factors of $T^{2r}$.
	Let $\varphi$ and $\sigma $ be two elements of $S_{2r}$ defined as follows
	$$\varphi := (1,2r)(2,2r-1)(3,2r-2)\dots(r, 2r-(r-1)) \hbox{, and }$$
	$$ \sigma := (1,2,3,\dots,r)(r+1, r+2, r+3,\dots,2r).$$
	Note that $\varphi^2 = \sigma ^r = 1$ and $\sigma^\varphi = \sigma^{-1}$. Let $n:= (a,b,1,\dots,1,b,a)\in T^{2r}$, that is, every entry of $n$ is the identity element except for entries $1$, $2$, $2r-1$, and $2r$ which are $a$, $b$, $b$ and $a$ respectively.
	
	Now let $N:= T^{2r}$ and define the group $G := N \rtimes \langle \sigma, \varphi \rangle \leq T \wr S_{2r}$, so that $G \cong T^{2r} \rtimes D_r$. Finally, define two elements $s$ and $g$ of $G$ as $s := \varphi$, and $g := n\sigma$ and let $S : =\langle s \rangle$ and   $\Gamma := $ \rm{Cos}$(G, S, g)$.
\end{Construction}

\begin{Remark}\label{rem:simple}
 There are infinitely many simple groups $T$ fitting the requirements of Construction \ref{NonabelianUnorientedConst2r}. In particular, any nonabelian finite simple group can be generated by an involution $a$ and an element $b$ of prime order $p \geq 3$, see \cite{king2017generation}. In any such group the element $ab$ cannot be an involution since this would imply that $b^a=(ab)^2b = b$ and hence that $T=\langle b\rangle\rtimes \langle a\rangle$, contradicting the fact that $T$ is a nonabelian simple group. Since $ba=a(ab)a=(ab)^a$, the elements $ab$ and $ba$ have the same order. Thus the only extra requirement that needs checking  is that $|b| \neq |ab|$.  For an explicit example of such a group and generating set we may take $T = \PSL_{2}(p)$ for a prime $p \geq 5,$ and generators 
$$
a:=\begin{pmatrix} 0 & 1\\-1 & 0  \end{pmatrix} \hbox{ and } b:= \begin{pmatrix} 0 & 1\\-1 & 1  \end{pmatrix} \hbox{ modulo scalars. }
$$
In this case $a$ and $b$ have orders 2 and 3 respectively, while $ab$ has order $p$ (see \cite[Section 7.5]{coxeter2013generators}).   
\end{Remark}

Our aim is to show that each pair $(\Gamma, G)$ in Construction \ref{NonabelianUnorientedConst2r} lies in $\OG(4)$ and is basic of unoriented-cycle type, such that $G$ has a nonabelian socle with the largest possible number of simple direct factors. Since $r \geq 3$ can be chosen arbitrarily, this construction provides an infinite family of nonisomorphic  4-valent basic $G$-oriented graphs with unoriented cyclic normal quotients.
The most difficult part of the proof that $(\Gamma, G)\in\OG(4)$ is showing that $G=\langle g, s\rangle$, and hence that $\Gamma$ is connected, and we address this part separately in  the following lemma. 

\begin{Lemma}\label{gen}
	Let $G, s$ and $g$  (and hence also $r, n, N$, $\varphi$ and $ \sigma$) be as in Construction $\ref{NonabelianUnorientedConst2r}$. Then $\langle g, s \rangle = G$.
\end{Lemma}



\begin{proof} 
    We will show that $N \leq \langle g, s \rangle$ from which the result follows since $\langle g, s \rangle $ modulo $N$ is equal to $\langle \sigma, \varphi\rangle$.  We do this by identifying two particular nontrivial elements of the subgroup $C:= N\cap \langle g,s\rangle$ and proving that $C = N$.
    We represent the elements of $N = T_1 \times\dots \times T_{2r}$ as $2r$-tuples in the canonical way, and for  $m \in N$  and an integer $\ell \in [1,2r]$, we let $[m]_\ell$ denote the $\ell$th entry of $m$. In particular,  since $n = (a,b,1,\dots, 1, b, a) \in N$, we have $[n]_\ell = [n]_{2r+1-\ell}$ for each $\ell \in [1,2r]$,  and $[n]_1 = a, [n]_2 = b$, and $[n]_\ell = 1$ for $\ell \in [3, r]$; also $n^s=n^\varphi=n$. The elements we consider are
        $$
        X_1 = g^sg\quad \hbox{and}\quad X_2 =  (g^2)^sg^2, \quad\hbox{both in $\langle g,s\rangle$.}
        $$
        We  show these elements lie in $N$ and hence lie in $C$. 
        Then we deduce that $C$ projects onto each simple  direct factor of $N = T^{2r}$, and finally, a careful case analysis is given to show that $C$ cannot be a product of diagonal subgroups of  nontrivial  subproducts of $T^{2r}$, and hence $C =N$, by Lemma \ref{stripmethod}. 

        Note that $g^2 = n\sigma n\sigma =  n n^{\sigma^{-1}}\sigma^2$. Then, recalling that $n^s=n$ and $\sigma^s=\sigma^{-1}$, we have $X_1=g^sg=(n\sigma)^s(n\sigma)n\sigma^{-1}n\sigma = nn^{\sigma}\in N$, and
        $X_2 =  (g^2)^sg^2=(n n^{\sigma^{-1}}\sigma^2)^s ( n n^{\sigma^{-1}}\sigma^2)= nn^\sigma n^{\sigma^2} n^\sigma \in N$.
        As $2r$-tuples, the elements $X_1$ and $X_2$ have a slightly different form when $r = 3$ from larger values of $r >3$.
        If $r = 3$ then, recalling that $a^2 = 1$ and $n^s=n$, we have  
	\begin{align*}
	X_1 &=  nn^\sigma = (a,b,1,1,b,a)\cdot (1,a,b,a,1,b)= (a,ba,b,a,b,ab), \hbox{ and }\\
	X_2 &=  nn^\sigma \cdot n^{\sigma^2}\cdot n^\sigma = (a,ba,b,a,b,ab) \cdot (b,1,a,b,a,1)\cdot (1,a,b,a,1,b)\\
        &= (ab, b, bab, aba,ba,ab^2),
	\end{align*}
    while if $r\geq 4$, then
        \begin{align*}
	X_1 &=  nn^\sigma = (a,b,1,\dots,1,b,a)\cdot (1,a,b, 1,\dots,1,a,1,\dots,1, b)\\
        &=  (a, ba, b, 1, \dots, 1,  a, 1, \dots,  1, b, ab),
 	\end{align*}
    where in the last expression each of  the two strings of $1$'s has length $r-3$, and 
    \begin{align*}
	X_2 &= nn^\sigma \cdot n^{\sigma^2}\cdot n^\sigma \\
        &=  (a, ba, b, 1, \dots, 1,  a, 1, \dots,  1, b, ab) \cdot (1, 1, a, b,1,\dots,1, b,a,1\dots,1)\cdot (1,a,b, 1,\dots,1,a,1,\dots,1, b)\\
        &= (a, ba^2, bab,b,1,\dots,1, aba, a, 1,\dots,1,b,ab^2),
	\end{align*}
    and here  in the last expression each of  the two strings of $1$'s has length $r-4\geq0$. For any $r\geq3$, we see that $\pi_{r+1}(C)$ contains  $\pi_{r+1}(\langle X_1, X_2\rangle)
    = \langle a,aba\rangle = \langle a,b\rangle \cong T$, and hence $\pi_{r+1}(C)= T_{r+1}$. Further $\langle g,s\rangle$ acts transitively by conjugation on the $2r$ simple direct factors of $N$, and 
    normalises $C=N\cap \langle g,s\rangle$. This implies that $\pi_i(C)= T_i$ for all $i$. Thus $C$ is a subdirect subgroup of  $N = T^{2r}$ and property (a) of Lemma \ref{stripmethod} holds. To check that property (b) of Lemma \ref{stripmethod} holds for $C$, recall that, by Proposition \ref{strips}, $C = \prod_{k 
		\in K}  A_k$, where each $A_k$ is a full strip, and if $S_k := \Supp(A_k)$, then $S_k \cap S_j = \emptyset$ whenever $k \neq j$. Our task is to show that each $S_k$ has size 1, implying that Lemma \ref{stripmethod}(b) holds. Further, the group $\langle g,s\rangle$ acts transitively by conjugation on the set of strips $\{ A_k \mid k\in K\}$ in $C$, and hence the supports $S_k$ all have the same size.  To verify property (b) of Lemma \ref{stripmethod}, suppose to the contrary that some direct factor $F$ of $C$ is a nontrivial full strip of $N$. Then each of the strips $A_k$ is nontrivial, and we may assume that $T_1$ lies in the support of $F$. By considering the orders of the elements in the entries of $X_1$ in light of Lemma \ref{stripAuto}, we conclude that $F$ is a full diagonal subgroup of the subproduct $T_1\times T_{r+1}$. (This is because if $T_y$ and $T_z$ lie in a single $S_k$ then the $y$th and $z$th entries of any element of $C$ must be equivalent modulo $\Aut(T)$.)
        It follows from Lemma  \ref{stripAuto} that, for each element of $C$, the entries in positions $1$ and $r+1$ must have the same order.  In particular, the orders of $[X_2]_1 = a$ and $[X_2]_{r+1}=aba=b^a$ should be equal, which is not the case. This contradiction implies that each $S_k$ has size $1$ and hence $C=N$ and $\langle g,s\rangle = G$, completing the proof.
        \end{proof}

Now we show that the pairs defined in Construction \ref{NonabelianUnorientedConst2r} are basic of unoriented-cycle type. 

\begin{Lemma}
	Let $(\Gamma, G)$ be as in Construction $\ref{NonabelianUnorientedConst2r}$. Then the following hold \begin{enumerate}[(a)]
		\item $(\Gamma, G) \in \OG(4)$,
		\item $(\Gamma, G)$ is basic of unoriented-cycle type.
	\end{enumerate}
\end{Lemma}

\begin{proof}
	To prove part (a), we verify properties (i) - (iv) of Lemma \ref{CosetOriented}. 
	Recall that $s = \varphi$, $g = n \sigma$, $n^\varphi = n$ and $\sigma ^\varphi = \sigma^{-1}$. Since $s^g = \sigma^{-1}n^{-1}\varphi n \sigma$, which is equal to $\sigma^{-1}\varphi\sigma  = \varphi \sigma^2$ modulo $N$, while $s$ is equal to $\varphi$ modulo $N$, it follows that $s^g \neq s$ and hence that properties (i) and (iii) hold.
	Next we  show  that property (ii) holds, that is, that $g^{-1}\not\in SgS = \{g, sg, gs, sgs\}$. It is easy to see that $g^{-1} \neq g, sg, gs$ by considering these elements modulo $N$, so we only need to check that $g^{-1} \neq sgs$.
	For this, notice that if $sgs = g^{-1}$ then $sgsg = 1$. But $sgsg = \varphi n\sigma \varphi n \sigma = n n^{\sigma} \neq 1 $ since $n^{\sigma} \neq n^{-1}$. Hence $sgs  \neq g^{-1}$ and (ii) holds. Finally, property (iv) holds by Lemma \ref{gen}, and hence $(\Gamma, G) \in \OG(4)$ by Lemma \ref{CosetOriented}, proving part (a).
	
	For part (b) we argue as follows. Firstly, $N$ is a minimal normal subgroup of $G$ since $G$ acts transitively by conjugation on the simple direct factors of  $N$. Moreover, $C_G(N) =1$ as $\langle g,s\rangle$ acts faithfully by conjugation on $N$, and hence $N$ is the unique minimal normal subgroup of $G$. We show that $(\Gamma, G)$ is basic of unoriented-cycle type with $\Gamma_N$ a $G$-unoriented cycle by applying Lemma \ref{CosetCycleType}(b).
	
	
	Consider the two distinct $SN$-cosets, $SN\sigma$ and $SN\sigma^{-1}$. Since $g = n\sigma \in N\sigma$, we have $Sg = \{g, sg\} \subset (SN)\sigma$. On the other hand, since
	$\sigma ^\varphi = \sigma^{-1}$ and $SN =  N\langle \varphi \rangle$, it follows that  $gs =n\sigma \varphi = \varphi n\sigma^{-1} \in  (SN)\sigma^{-1}$, and so $Sgs = \{gs, sgs\} \subset SN\sigma^{-1}$. In particular half of the elements of $SgS$ are contained in $SN\sigma$ and the other half are contained in $SN\sigma^{-1}$. 	
	Similarly, $g^{-1}$ = $\sigma^{-1}n^{-1} \in N\sigma^{-1}$ so $Sg^{-1} = \{g^{-1}, sg^{-1}\} \subset (SN)\sigma^{-1}$, and $g^{-1}s = \sigma^{-1}n^{-1}\varphi = \varphi (n^{-1})^{\sigma\varphi} \sigma \in  (SN)\sigma$.  So $Sg^{-1}s = \{g^{-1}s, sg^{-1}s\} \subset SN\sigma$. So again half of the elements of $Sg^{-1}S$ are contained in $SN\sigma$ and the other half are contained in $SN\sigma^{-1}$.	
	So the conditions of Lemma \ref{CosetCycleType}(b) hold, and $(\Gamma, G)\in \OG(4)$ is basic of unoriented-cycle type.
\end{proof}

At this point we can state and prove the following theorem, which gives an affirmative answer to \cite[Problem 1]{al2017cycle}.

\begin{theorem}\label{thm:problemsolved}
The number of unoriented cyclic normal quotients that a basic pair $(\Gamma, G)\in \OG(4)$ can have can be unboundedly large.
\end{theorem}
\begin{proof}
    First notice that if $(\Gamma, G)$ is as in Construction \ref{NonabelianUnorientedConst2r} and we take a divisor $d$ of $r$ then we can take a subgroup $M_d \leq G$ where $M_d = N \rtimes \langle \sigma^d \rangle$. It is easy to see that $M_d$ is normal in $G$ since $N$ is normal and $\langle \sigma^d \rangle$ is normalised by $\langle \varphi, \sigma \rangle$. Since $\Gamma_N$ is an unoriented cycle of length $r$, $\Gamma_{M_d}$ will be an unoriented cycle of length $r/d$ (provided $d \neq r/2)$. Hence if we apply Construction \ref{NonabelianUnorientedConst2r}, with $r = p^m$ where $p$ is an odd prime, and $m$ arbitrarily large, we can construct a basic pair $(\Gamma, G)$ with arbitrarily many unoriented cyclic normal quotients.
\end{proof}

\subsubsection{The Final Construction}
Our final construction provides an infinite family  of graph-group pairs $(\Gamma, G) \in \OG(4)$ of the kind described in Theorem \ref{CycleMainTheorem} Case 1(b)i. with $k = r2^r$. In this case, our construction only produces examples having $r = 3$ and $k = 24$.

Similarly to the groups and generating pairs used in Construction $\ref{NonabelianUnorientedConst2r}$, there are infinitely many simple groups $T$ having generating pairs $(a,b)$ of the kind described in our following construction (see the discussion in Remark~\ref{rem:simple} for examples of such groups and generating pairs).

\begin{Construction}\label{NonabelianOrientedConst24}
	Let $T$ be a nonabelian simple group generated by two elements $a$ and $b$, where $a$ is an involution, $b$ has order $3$.
	Consider the group $T \wr S_{24}$ with $S_{24}$ acting by permuting the ${24}$ simple direct factors of $T^{24}$. Let $\Omega := \{1,\dots,24\}$ and let $I_1 := \{1,\dots,8\}$, $I_2 := \{9,\dots,16\}$ and $I_3 := \{17,\dots,24\}$ partition $\Omega$ into three subsets of size $8$.
	Let $\varphi_0$ and $\rho $ be elements of $S_{24}$ defined as follows
	$$\varphi_0 := (1,2)(3,4)(5,6)(7,8)\cdot(9,10)(11,12)(13,14)(15,16)\cdot(17,18)(19,20)(21,22)(23,24),$$
	$$\rho := (2,4,8)(3,5,7)\cdot(10,12,16)(11,13,15) \cdot(18,20,24)(19,21,23).$$
	Next define $\varphi_1$ and $\varphi_2$ as 
	\begin{align*}
	    \varphi_1 &:= \varphi_0^\rho \\&= (1,4)(2,3)(5,8)(6,7)\cdot(9,12)(10,11)(13,16)(14,15) \cdot(17,20)(18,19)(21,24)(22,23)\\
    \varphi_2 &:= {\varphi_0}^{\rho^2} \\& \hspace{-1mm}= (1,8)(2,7)(3,6)(4,5)\cdot(9,16)(10,15)(11,14)(12,13) \cdot(17,24)(18,23)(19,22)(20,21).
\end{align*}
	Notice that $\varphi_0$, $\varphi_1$ and $\varphi_2$ are pairwise commuting involutions which are cyclically permuted by conjugation by $\rho$, and that $\langle \varphi_0, \varphi_1, \varphi_2\rangle$ acts semiregularly with $3$ obits on $\Omega$, namely $I_1, I_2$, and $I_3$.
	Now define two more elements of $S_{24}$, $\theta$ and $\sigma$ by
	$$\theta: i \mapsto i+8 \hbox{ mod } 24 \hbox{, and}$$
	$$
    \sigma: i \mapsto i^{\rho\theta}, \hbox{ that is } \sigma = \rho\theta.
    $$
	Note that $[\rho, \theta] =1$ which implies that $\sigma$ has order $3$; moreover $[\theta, \varphi_i] =1$ for each $i \in \{0,1,2\}$, which implies that $\varphi_i^\sigma = \varphi_{i+1 \emph{ mod } 3}$. 
	
	Let  $n$ be an element of $T^{24}$ defined as 
	$$n:= (a,a,a,a,1,1,1,1,b,b,b,b,1,1,1,1,b,b,b,b,1,1,1,1).$$
	
	Further, let $N:= T^{24}$ and define the group $G := N \rtimes \langle  \varphi_0, \varphi_1, \varphi_2 ,\sigma\rangle \leq T \wr S_{24}$, so that $G \cong (T^{24} \rtimes C_2^3) \rtimes C_3$. 
	Finally, let $g= n\sigma \in G$, let $S : =\langle \varphi_0, \varphi_1, \varphi_2 \rangle$ and let $\Gamma := $ \rm{Cos}$(G, S, g)$.
\end{Construction}

 Once again,  the most difficult aspect of showing that our construction is valid is showing that the graph $\Gamma$ is connected. We therefore begin this process with the following lemma.
 
\begin{Lemma}\label{gen24}
	Let $G, S$ and $g$ be as in Construction $\ref{NonabelianOrientedConst24}$. Then $\langle S, g \rangle = G$.
\end{Lemma}
\begin{proof}
	Notice that $\langle S, g \rangle$ mod $N$ is equal to $\langle \varphi_0, \varphi_1, \varphi_2, \sigma \rangle$ hence we only need to show that $N \leq \langle S, g \rangle$. 
	
	To this end take an element $X_0$ of $\langle S, g \rangle$ defined as $X_0 := \varphi_0 \cdot \varphi_0^{g^3}$. In order to explicitly compute this element it is important to note that the element $n$ commutes with both $\varphi_0$ and $\varphi_1$. In particular $\varphi_0^g = \varphi_0^{n\sigma}  = \varphi_0^{\sigma} = \varphi_1$, and hence $\varphi_0^{g^2} = \varphi_1^g = \varphi_1^{n\sigma} = \varphi_1^\sigma = \varphi_2$. We will also use the fact that  $$n^{-1}:= (a,a,a,a,1,1,1,1,b^{-1},b^{-1},b^{-1},b^{-1},1,1,1,1,b^{-1},b^{-1},b^{-1},b^{-1},1,1,1,1) \hbox{, while }$$
	$$n^{\varphi_2} = (1,1,1,1,a,a,a,a, \hspace{4mm}1,1,1,1,b,b,b,b, \hspace{4mm} 1,1,1,1,b,b,b,b). $$
	Now we may compute $X_0$ as follows
	\begin{align*}
	X_0 &= \varphi_0 \cdot \varphi_0^{g^3} = \varphi_0\cdot \varphi_2^g = \varphi_0\cdot \varphi_2^{n\sigma}\\
	&= \varphi_0\cdot (\sigma^{-1}n^{-1}\varphi_2n\sigma) \\
	&= \varphi_0\cdot (\sigma^{-1}n^{-1}n^{\varphi_2}\varphi_2\sigma)\\
	& =  \varphi_0\cdot((n^{-1})^\sigma n^{\varphi_2\sigma}\varphi_2^\sigma)  = \varphi_0 ((n^{-1})^\sigma n^{\varphi_2\sigma}) \varphi_0\\
	& = (n^{-1}n^{\varphi_2})^{\sigma\varphi_0}\\
	& = (a,a,a,a,a,a,a,a, \hspace{4mm} b^{-1},b^{-1},b^{-1},b^{-1},b,b,b,b, \hspace{4mm}b^{-1},b^{-1},b^{-1},b^{-1},b,b,b,b)^{\sigma\varphi_0}\\
	& = (b^{-1}, b, b, b^{-1},b^{-1} , b, b, b^{-1}, \hspace{4mm} a,a,a,a,a,a,a,a, \hspace{4mm} b^{-1}, b, b, b^{-1},b^{-1} b, b, b^{-1})^{\varphi_0}\\
	& = (b, b^{-1}, b^{-1}, b,b , b^{-1}, b^{-1}, b,\hspace{4mm} a,a,a,a,a,a,a,a,\hspace{4mm} b, b^{-1}, b^{-1}, b,b,  b^{-1}, b^{-1}, b).
	\end{align*}
	
	Now define a subset $C$ of $\langle S, g \rangle$ as $C := \{X_0^{g^i} : i \in [0,5]\}$, and for each $i \in [0,5]$, set $X_i := X_0^{g^i}$. To compute $X_i$ given $X_{i-1}$ we simply conjugate $X_i$ by $n$ and then permute the direct factors according to $\sigma$. Note that conjugation by $n$ only affects entries in the range $[1, 4] \cup [9,12] \cup [16,20]$. 
	
	Thus for instance we can compute $X_1$ as follows 
	\begin{align*}
	X_1 &= (b, b^{-1}, b^{-1}, b,b , b^{-1}, b^{-1}, b,\hspace{4mm} a,a,a,a,a,a,a,a,\hspace{4mm} b, b^{-1}, b^{-1}, b,b,  b^{-1}, b^{-1}, b)^{n\sigma}\\
	& \hspace{-8mm}= (b^a, (b^{-1})^a, (b^{-1})^a, b^a,b , b^{-1}, b^{-1}, b,\hspace{3mm} a^b,a^b,a^b,a^b,a,a,a,a,\hspace{3mm} b, b^{-1}, b^{-1}, b,b,  b^{-1}, b^{-1}, b)^\sigma\\
	& \hspace{-8mm}= (b,b,b^{-1},b^{-1},b^{-1},b^{-1},b,b,\hspace{3mm} b^a,b,b^{-1},(b^{-1})^a,(b^{-1})^a,b^{-1},b,b^a,\hspace{3mm}a^b, a,a,a^b,a^b,a,a,a^b).
	\end{align*}
	Proceeding in this way we may compute the rest of the elements in $C$. In Table \ref{C24table} we give an entrywise description of the elements of $C$. The entries in each row can be computed from the row appearing immediately before it by first conjugating entries in $\{1,2,3,4\}$ by $a$, and entries in $\{9,10,11,12\}\cup\{17,18,19,20\}$ by $b$, and then applying the permutation $\sigma$ on the entries, as was done for computing $X_1$ from $X_0$.

	    \begin{table}[H]
		\renewcommand{\arraystretch}{2}
		\centering
		\footnotesize\setlength{\tabcolsep}{2pt}
		\resizebox{\linewidth}{!}{
		\begin{tabular}{|l@{\hspace{11pt}} |*{8}{c}|*{8}{c}|*{8}{c}|}
					\hline

				$i$ & 1 & 2 & 3 & 4 & 5 & 6 & 7 & 8 & 9 & 10 & 11 & 12 & 13 & 14 & 15 & 16 & 17 & 18 & 19 & 20 & 21 & 22 & 23 & 24 \\
				\hline
				$X_0$
				&$b$&$ b^{-1}$&$ b^{-1}$&$ b$&$b $&$ b^{-1}$&$ b^{-1}$&$ b$&$ a$&$a$&$a$&$a$&$a$&$a$&$a$&$a$&$b$&$ b^{-1}$&$ b^{-1}$&$ b$&$b$&$  b^{-1}$&$ b^{-1}$&$ b$\\
			     $X_1$
				&$b$&$b$&$b^{-1}$&$b^{-1}$&$b^{-1}$&$b^{-1}$&$b$&$b$&$ b^a$&$b$&$b^{-1}$&$(b^{-1})^a$&$(b^{-1})^a$&$b^{-1}$&$b$&$b^a$&$a^b$&$ a$&$a$&$a^b$&$a^b$&$a$&$a$&$a^b$ \\
				
				 $X_2$
				&$a^{b^2}$&$a^b$&$a$&$a^b$&$a^b$&$a$&$a^b$&$a^{b^2}$&$ b^a$&$b$&$b$&$b^a$&$(b^{-1})^a$&$b^{-1}$&$b^{-1}$&$(b^{-1})^a$&$b^{ab}$&$b^a$&$b$&$b$&$b^{-1}$&$b^{-1}$&$(b^{-1})^a$&$(b^{-1})^{ab}$ \\
				
			 $X_3$
				&$b^{ab^2}$&$(b^{-1})^{ab}$&$(b^{-1})^a$&$b^{ab}$&$b$&$b^{-1}$&$b^{-1}$&$b$&$  a^{b^2a}$&$a^{b^2}$&$a^b$&$a^{ba}$&$a$&$a$&$a^b$&$a^{ba}$&$ b^{ab}$&$(b^{-1})^a$&$b^{-1}$&$b$&$b$&$b^{-1}$&$(b^{-1})^a$&$b^{ab}$\\ 
				
				 $X_4$
				&$	b^{ab^2}$&$b^{ab}$&$(b^{-1})^a$&$(b^{-1})^{ab}$&$b^{-1}$&$b^{-1}$&$b$&$b$&$b^{ab^2a}$&$b$&$b^{-1}$&$(b^{-1})^{aba}$&$b^{-1}$&$b^{-1}$&$b$&$b^{aba}$&$a^{b^2ab}$&$a^{ba}$&$a^b$&$a^{b^3}$&$a^{b^2}$&$a$&$a$&$a^{bab}$\\
				 $X_5$                                                 
				&$a^{b^2ab^2}$&$a^{bab}$&$a$&$a^{bab}$&$a^{b^2}$&$a$&$a^{b^2}$&$a^{b^4}$&$b^{ab^2a}$&$b$&$b$&$b^{aba}$&$b^{-1}$&$b^{-1}$&$b^{-1}$&$(b^{-1})^{aba}$&$b^{ab^2ab}$&$b^{aba}$&$b$&$b$&$b^{-1}$&$b^{-1}$&$b^{-1}$&$(b^{-1})^{abab}$\\	\hline		
			\end{tabular}
   }
    \caption{Entrywise description of elements in the set $C$.}\label{C24table}
	\end{table}

		\renewcommand{\arraystretch}{1.5}

Now define the group $K := N \cap \langle S , g \rangle$. Our aim is to show that $K = N$, from which our main claim that $N \leq \langle S, g \rangle$ follows immediately. We have already established that $X_0 \in N$ and hence $X_0 \in K$. Note that $K$ is normalised by $\langle S , g \rangle$ since $N\unlhd G$.  Hence $X_0^{g^i}\in K$ for each $i$, so that $C \subseteq K$, whence $\langle C \rangle \leq K$. Our plan is to show that $K = N$ by appealing to Lemma \ref{stripmethod}.

	We first claim that $K$ is a subdirect subgroup of $N$. To see this it suffices to consider only the projections of the elements $X_0$, $X_1$, and $X_2$ onto the direct factors of $N$.  For each integer $i \in [1, 24]$ let $\pi_i$ denote the projection map $\pi_i: N \rightarrow T_i$ given by $(g_1, \dots, g_{24}) \mapsto g_i$. In fact it is easy to see that 
	$$\pi_i(\langle X_1,X_2\rangle) \supseteq \langle a, b \rangle  = T_i \hbox{ for each } i \in [1,8], \hbox{ while }$$
	$$\pi_i(\langle X_0,X_1\rangle) \supseteq \langle a, b \rangle  = T_i \hbox{ for each } i \in [9,24].
        $$
	Hence $K$ projects onto each simple direct factor of $N$. Thus property (a) of Lemma \ref{stripmethod} holds for $K$.
    To check that property (b) of Lemma \ref{stripmethod} holds for $K$, recall that, by Proposition \ref{strips}, $K = \prod_{k 
	\in \mathcal{K}}  A_k$, where each $A_k$ is a full strip, and if $S_k := \Supp(A_k)$, then $S_k \cap S_j = \emptyset$ whenever $k \neq j$. Our task is to show that each $S_k$ has size 1, implying that Lemma \ref{stripmethod}(b) holds. 
  
    Note that $\langle S, g\rangle$ acts transitively by conjugation on the $24$ simple direct factors of $N$. 
    Then, since the group $\langle S, g\rangle$ normalises $K$, it also acts transitively by conjugation on the set of strips $\{ A_k \mid k\in \mathcal{K}\}$ occurring in the direct decomposition of $K$. It follows that the supports $S_k$ all have the same size. 
	Now suppose for a contradiction that some simple direct factor $F$ of $K$ is a nontrivial full strip of $N$. Then each of the strips $A_k$ is nontrivial, and we may assume that $T_1$ lies in the support of $F$. In other words, $F$ is a diagonal subgroup of a subproduct  $\prod_{j \in J}^{} T_j$  where $1 \in J$ and $|J|\geq 2$. 
 
    This means in particular that, for each element $X\in K$, the order of the projection $\pi_j(X)$ is independent of the choice of $j\in J$. 
     We will let $I_1 := [1,8], I_2:= [9,16], $ and $I_3:= [17,24]$, so that these three subsets partition $ [1,24]$. Now notice that for the elements $X_0, X_1, X_2 \in K$, the orders of the elements in the entries are as follows:
	\begin{align*}
	|\pi_i(X_0)|  &= 3, \hbox{ if } i \in I_1\cup I_3 \hbox{, and } |\pi_i(X_0)|  = 2, \hbox{ if } i \in I_2,\\
	|\pi_i(X_1)|  &= 3, \hbox{ if } i \in I_1\cup I_2 \hbox{, and } |\pi_i(X_1)|  = 2, \hbox{ if } i \in I_3,\\
	|\pi_i(X_2)|  &= 3, \hbox{ if } i \in I_2 \cup I_3 \hbox{, and } |\pi_i(X_2)|  = 2, \hbox{ if } i \in I_1.
	\end{align*}
	Lemma \ref{stripAuto} then implies that since $1\in J$ we must have $J \subseteq I_1$.


	
 Next consider the actual $I_1$-entries in the elements $X_0$ and $X_1$. It is clear by inspection that $J \subseteq \{1,3,6,8\}$, for otherwise there would be an automorphism $\psi \in \Aut(T)$ which maps $b$ to both $b$ and $b^{-1}$, which is a contradiction since $|b| = 3$.
 We will now show that none of the 2-subsets of $\{1,3\}$, $\{1,6\}$, or $\{1,8\}$ is contained in $J$, contradicting that assumption that $|J|\geq 2$.
		
	
	If $\{1,8\}\subseteq J$, then since $(\pi_1(X_0), \pi_8(X_0)) = (b,b)$ while $(\pi_1(X_3), \pi_8(X_3)) = (b^{ab^2},b)$, we would have an automorphism $\psi \in \Aut(T)$ which maps both $b$ and $b^{ab^2}$ to $b$, implying that $b = b^{ab^2}$, which is a contradiction since $a$ and $b$ do not commute. A nearly identical argument shows that $\{1,6\} \nsubseteq J$ since $(\pi_1(X_0), \pi_6(X_0)) = (b,b^{-1})$ while $(\pi_1(X_3), \pi_6(X_3)) = (b^{ab^2},b^{-1})$, so  $\{1,6\} \subseteq J$ would also imply that $b = b^{ab^2}$, and we would have a contradiction.
	
	Finally, if $\{1,3\} \subseteq J$, then since $(\pi_1(X_2), \pi_3(X_2)) = (a^{b^2},a)$ while $(\pi_1(X_5), \pi_3(X_5)) = (a^{b^2ab^2},a)$, there would be an automorphism of $T$ which maps both $a^{b^2}$ and $a^{b^2ab^2}$ to $a$, which implies that $a^{b^2} = a^{b^2ab^2}$, or equivalently, $a = a^{b^2a}$. This implies that $a = a^{b^2}$ and hence that $a$ commutes with $b^2$ and hence with $b$, which is a contradiction. Thus $\{1,3\} \nsubseteq J$. 
    So $J$ cannot contain more than one element and hence no direct factor of $K$ is a nontrivial full strip of $N$. Therefore $K  = N$ and so $\langle S, g \rangle = G$ as claimed. 
\end{proof}

We are now able to show that the pair given in Construction \ref{NonabelianOrientedConst24} is a basic pair of oriented-cycle type. 
\begin{Lemma}
	Let $(\Gamma, G)$ be as in Construction $\ref{NonabelianOrientedConst24}$. Then the following hold \begin{enumerate}[(a)]
		\item $(\Gamma, G) \in \OG(4)$,
		\item $(\Gamma, G)$ is basic of oriented-cycle type.
	\end{enumerate}
\end{Lemma}
\begin{proof}
	For (a) we check that conditions (i)-(iv) of Lemma \ref{CosetOriented} hold. For (i), take the element $h:= (a,1,\dots, 1) \in N$.  Now taking any $s\in S$ we see that $h^{-1}sh = h^{-1}h^s s$. Since $S$ acts regularly on the first 8 entries of $h$, it follows that for $s \neq 1$, $h^s \neq h$ and so $h^{-1}h^s$ is a nontrivial element of $N$. In particular $s^h = h^{-1}h^s s \notin S$ and so $S^h \cap S = 1$. Thus  $S$ is core-free in $G$.
	
	For (ii), note that $g^{-1} = \sigma^{-1}n^{-1} = (n^{-1})^\sigma \sigma^{-1}$, but every element in $SgS$ is of the form $s_1n\sigma s_2 = s_1n s_2^{\sigma^{-1}}\sigma$, for some $s_1, s_2 \in S$. Therefore, since $g^{-1}$ is not equal to any such element modulo the normal subgroup $\tilde{N} = N \rtimes S$, it follows that $g^{-1} \notin SgS$. 
	
	For (iii) we compute $S^g$. We have already seen that $\varphi_0^g = \varphi_1$ and that $\varphi_1^g = \varphi_2$. On the other hand $$\varphi_2^g = \sigma^{-1}n^{-1}\varphi_2n\sigma =
	\sigma^{-1}n^{-1}n^{\varphi_2}\varphi_2\sigma = (n^{-1})^\sigma  n^{\varphi_2\sigma} \varphi_2^\sigma =(n^{-1}  n^{\varphi_2})^{\sigma} \varphi_0.$$
	Now since $n^{\varphi_2} \neq n$ it follows that $n^{-1}  n^{\varphi_2}$ is a nontrivial element of $N$, which implies that $(n^{-1}  n^{\varphi_2})^{\sigma}$ is also a nontrivial element of $N$. Thus $\varphi_2^g \notin S$. Since $\langle \varphi_0, \varphi_1 \rangle \leq S\cap S^g$ it follows that $\langle \varphi_0, \varphi_1 \rangle =  S\cap S^g$, so  $|S: S \cap S^g| =2$   and condition (iii) holds. Condition (iv) holds by Lemma \ref{gen24}, so $(\Gamma, G) \in \OG(4)$. 
	
	 Next, since $G$ acts transitively by conjugation on the simple direct factors of  $N$, we know that $N$ is minimal normal in $G$, and since $C_G(N) =1$, $N$ is the unique minimal normal subgroup of $G$.
	
	Finally note, that $$SgS = Sn\sigma S = Sn S^{\sigma^{-1}} \sigma = SnS\sigma \subset SNS\sigma = SN\sigma, \hbox{ and }$$
	$$Sg^{-1}S = S\sigma^{-1}n^{-1}S = S(n^{-1})^\sigma S^{\sigma} \sigma^{-1} \subset SNS \sigma^{-1} = SN \sigma^{-1}.$$
	
	Hence $SgS$ and $Sg^{-1}S$ lie in two distinct $SN$ cosets, namely $SN\sigma$ and $SN\sigma^{-1}$, so $(\Gamma, G)$ is basic of oriented-cycle type by Lemma \ref{CosetCycleType}.
\end{proof}

\section*{Acknowledgements}
The authors were grateful for the opportunity to participate in
the Tutte Memorial MATRIX retreat in 2017 where this work began. The first author was supported by an Australian Government Research Training Program (RTP) Scholarship, and was also supported by the University of Western Australia which funded his visit to Perth where some of this work was completed. The second author acknowledges Australian Research Council Funding
of DP160102323.  

We would also like  to thank Alex Ghitza for leading us to a reference which appears in Remark \ref{ArtinRemark}, and  Primo{\v{z}} Poto{\v{c}}nik for a very useful conversation in which he outlined the proof of Lemma \ref{Abelianstabiliser}. We also acknowledge the use of \textsc{Magma} \cite{MR1484478} for testing theories and constructing examples.


\end{document}